\documentclass[12pt,oneside,openany]{article}
\usepackage[pdftex]{graphicx}
\usepackage{latexsym}
\usepackage{url}
\usepackage{amsmath}
\usepackage{amssymb} 
\usepackage{makeidx}
\usepackage[top=1in, bottom=1in, left=1in, right=1in]{geometry}

\def\R{{\mathbb R}}

\newtheorem{theorem}{Theorem}   
\newtheorem{cor}[theorem]{Corollary}

\newtheorem{corr}[theorem]{Corollary}

\newenvironment{proof}{\medskip \par \noindent {\bf Proof}\ }{\hfill 
$\Box$ 
                       \medskip \par}

\def\R{{\mathbb R}}% Reals
\makeindex

\begin{document}
\overfullrule=0pt
\baselineskip=18pt
\baselineskip=24pt
\font\tfont= cmbx10 scaled \magstep3
\font\sfont= cmbx10 scaled \magstep2
\font\afont= cmcsc10 scaled \magstep2
\title{\tfont Functional analysis behind a Family of Multidimensional Continued Fractions:\\Part II }
\bigskip
\author{Ilya Amburg\footnote{Center for applied Mathematics, Cornell University, Ithaca, NY 14853; \texttt{ia244@cornell.edu}}\mbox{  } and Thomas Garrity\footnote{Williams College Department of Mathematics, Williamstown, Ma  01267; \texttt{tgarrity@williams.edu}}}

\date{}
\maketitle

\begin{abstract}

This paper is a direct continuation of ''Functional analysis behind a Family of Multidimensional Continued Fractions: Part I,''  in which we started the exploration of the functional analysis behind the transfer operators for triangle partition maps, a family that includes many, if not most, well-known multidimensional continued fraction algorithms. 
This allows us now to  find eigenfunctions of eigenvalue 1 for transfer operators associated with select triangle partition maps on specified Banach spaces. We proceed to prove that the transfer operators, viewed as acting on one-dimensional families of Hilbert spaces, associated with select triangle partition maps are nuclear of trace class zero. We finish by deriving Gauss-Kuzmin distributions associated with select triangle partition maps.

\end{abstract}

\section{Introduction}
For a general introduction to triangle partition maps and their associated transfer operators, please see the introduction to \cite{amburg-Garrity I}.  In particular, we will be using the notation of that paper.  Here we are interested in questions about the spectrum of the transfer operators.  For the traditional continued fraction algorithm and the associated transfer operator of the Gauss map, there is a rich history of studying such questions.  

Each multidimensional continued fraction in the family of triangle partition maps provides a description of each point $(\alpha, \beta)$ in $\triangle =\{(x,y)\in \R^2:0 < y < x<1\}$ via an infinite sequence of nonnegative integers (almost everywhere), in analog to a real number's continued fraction expansion.  Asking about the statistics of such sequences for each particular algorithm is natural and important.  For traditional continued fractions, these types of statistics fall under the rhetoric of Gauss-Kuzmin statistics.  The key technical tool is in finding the invariant measure of the Gauss map and then in understanding the spectrum of the corresponding transfer operator.  In particular, what is needed to be shown is that the largest eigenvalue of the operator is one, with a one-dimensional eigenspace.  Without using this language, this is what is done in theorem 33 in Khinchin's classic \cite{Khinchin}.  For us, we will be trying to generalize the work of Mayer and Roepstorff \cite{Mayer, Mayer-Roepstorff1988}.  The important question is first in finding the relevant Banach or Hilbert space of functions on the domain and then in determining the leading eigenvalue of the transfer operator.  These are delicate questions.

For traditional continued fractions, in \cite{Mayer-Roepstorff1988} , Mayer and Roepstorff do this for the Banach space of continuous functions on the unit interval and in \cite{Mayer}, do this for the Hilbert space $L^2[0,1]$.   Further, in the Hilbert space paper, they show that the transfer operator is nuclear of trace class zero.  (Their work is also described in \cite{Iosifescu-Kraaikamp1}.)

As discused in \cite{Garrity15}, there are significant technical reasons that the early work of Khinchin and the work of Mayer and Roepstorff  cannot be effortlessly extended to triangle partition maps.  One problem is to find the appropriate Banach spaces of functions for each map. We will see in section \ref{banach} that for some triangle partition maps there are natural Banach spaces of functions, that for $18$ of these  the largest eigenvalue of the associated transfer operator is one, and that for some of these  the dimension of the corresponding eigenspace is one. The Hilbert space approach is more challenging, as few of the natural choices for eigenfunctions with eigenvalue one will be in the $L^2$ space of square integrable functions on the domain. But we will see  in section \ref{hilbert} that a number of the triangle partition maps' transfer operators are nuclear operator of trace class zero when viewed as acting on naturally occurring one-dimensional families of Hilbert spaces. (We find that the need to think of the transfer operator as a family of nuclear operators of trace class zero as quite interesting.) In section \ref{GK}, we present the Gauss-Kuzmin statistics for select triangle partition maps. In section \ref{conclusion}, we discuss future directions for work.  Appendix \ref{app:TRIPmaps} is a table of the polynomial-behavior triangle maps, while appendix \ref{app:formoftransfer} is a table of the transfer operators for these polynomial-behavior triangle partition maps.

  For background on earlier work, see  Hensley\cite{ Hensley}, Iosifescu and Kraaikamp \cite{Iosifescu-Kraaikamp1}, Kesseb\"{o}hmer, Munday and Stratmann  \cite{Kessebohmer-Munday-Stratmann}, Khinchin \cite{Khinchin}, Rockett and Szusz \cite{Rockett-Szusz}, Schweiger \cite{Schweiger3} and Isola \cite{Isola02}.  In \cite{Garrity15}, this was done for the triangle map, where it was shown that much of the work of Mayer and Roepstorff \cite{Mayer, Mayer-Roepstorff1988} has nontrivial analogs for the triangle map. Building on an earlier version of \cite{Garrity15} , there is also the interesting work of Bonanno, Del Vigna and Munday \cite{Bonanno-Del Vigna-Munday}.  Further, there is the work of Fourgeron and Skripchenko \cite{Fougeron-Skripchenko} on the Lyapunov exponents of the triangle map.   again, the goal of this paper is to see which of these analogs hold for various members of the family of triangle partition maps.

Part of this paper stems from \cite{amburg}.

We would like to thank the referees for many useful comments.

\section{Leading Eigenvalues for Select Transfer Operators}
\label{banach}

All of the notation comes from \cite{amburg-Garrity I} .

It is well-known that the transfer operator for the Gauss map has leading eigenvalue one with a one-dimensional eigenspace when acting on the Banach space of bounded continuous functions on the unit interval, as in work of Mayer and Roepstorff  \cite{Mayer-Roepstorff1988}.  Similarly, in \cite{Garrity15}, it is shown that  the transfer operator for  $T_{e,e,e}$ has leading eigenvalue one with a one-dimensional eigenspace in the Banach space of continuous functions $f(x,y)$ on $\triangle $ where $|xf(x,y)|$ is bounded.  What about the other triangle partition maps?

We will show for $47$ triangle partition maps\footnote{In this paper we present new results for all maps but $T_{e,e,e},$ for which the analogous results are presented in \cite{Garrity15}; nonetheless, we include the results for $T_{e,e,e}$ for completeness.} that there are similar natural  Banach spaces, which suggests that the corresponding transfer operators may be good candidates for having largest eigenvalue one. For $18$ of these $47$ transfer operators, we will explicitly find an eigenfunction with eigenvalue one. When we restrict attention to those triangle partition maps which are known to be ergodic, we can finally conclude for these that the eigenspace with eigenvalue one is one-dimensional.

\subsection{Transfer Operators as Linear Maps on appropriate Banach Spaces}
The goal here is to find  Banach spaces for select $\mathcal{L}_{T_{\sigma,\tau_0,\tau_1}}$ which might be good analogs to the Banach space found in \cite{Garrity15}.  The Banach spaces will all have the following form.  Let $\mathcal{C}(\triangle)$ be the space of continuous functions on the open triangle $\triangle.$
Each of the Banach spaces will be 
$$W=\{f(x,y)\in \mathcal{C}(\triangle): \exists C \in \R : |g(x,y) f(x,y)|<C, \forall (x,y) \in \bigtriangleup \},$$
where $ g(x,y)\in \mathcal{C}(\triangle)$ is a fixed function. 

Each $W$ is a Banach space under norm
$$\|f(x,y)\|=\sup_{(x,y)\in \bigtriangleup} |g(x,y) f(x,y)|.$$

The following theorem is why these are natural Banach spaces to study:

\begin{theorem}
\label{theorem:Banach}
For functions $g(x,y)$ as shown in the table below, the transfer operators $\mathcal{L}_{T_{X}},$ with $X$ representing each of the permutation triplets appearing in the table, are linear maps from $W$ to $W.$ 

\end{theorem}

$$\begin{array}{l|c|c}
(\sigma, \tau_0, \tau_1)  & g(x,y)  & \mbox{Summand} \\
\hline
\hline
 (e,e,e) & x & \frac{x}{(k x+y+1)^2} \\ \hline
 (e,12,e) & (-x+y+1)^2 & -\frac{(x-y-1)^2}{(x k-y k-k+x-2 y-1)^2 (x k-y k-k-y-1)} \\ \hline
 (e,13,e) & (1-y)^2 & -\frac{(y-1)^2}{(y k-k+x-2) (y k-k+x-y-1)^2} \\ \hline
 (e,13,12) & (1-y)^2 & -\frac{(y-1)^2}{(y k-k+x-2) (y k-k+x-y-1)^2} \\ \hline
 (e,23,e) & x (1-y) & \frac{x-x y}{(k x-y+1) (k x+x-y+1)} \\ \hline
 (e,123,e) & (-x+y+1)^2 & -\frac{(x-y-1)^2}{(x k-y k-k+x-2) (x k-y k-k+2 x-y-2)^2} \\ \hline
 (e,132,e) & x (1-y) & \frac{x-x y}{(y k-k-x) (y k-k-x+y-1)} \\ \hline
 (e,132,12) & (1-y)^2 & -\frac{(y-1)^2}{(y k-k-x)^2 (y k-k-x+y-1)} \\ \hline
 (12,e,12) & x^2 & \frac{x^2}{(k x+y+1) (k x+x+y)^2} \\ \hline
 (12,12,12) & -x+y+1 & \frac{-x+y+1}{(-x k+y k+k+y+1)^2} \\ \hline
 (12,13,e) & (1-y)^2 & -\frac{(y-1)^2}{(y k-k+x-2) (y k-k+x-y-1)^2} \\ \hline
 (12,13,12) & (1-y) (-x+y+1) & \frac{(x-y-1) (y-1)}{(y k-k+x-2) (y k-k+x-y-1)} \\ \hline
 (12,23,12) & x^2 & \frac{x^2}{(k x+x-y+1) (k x+2 x-y)^2} \\ \hline
 (12,123,12) & (1-y) (-x+y+1) & \frac{(x-y-1) (y-1)}{(x k-y k-k+x-2) (x k-y k-k+y-1)} \\ \hline
 (12,132,e) & (1-y)^2 & -\frac{(y-1)^2}{(y k-k-x)^2 (y k-k-x+y-1)} \\ \hline
 (12,132,12) & (1-y)^2 & -\frac{(y-1)^2}{(y k-k-x)^2 (y k-k-x+y-1)} \\ \hline
 (13,e,13) & x^2 & \frac{x^2}{(k x+y+1) (k x-x+y+1)^2} \\ \hline
 (13,e,123) & x^2 & \frac{x^2}{(k x+y+1) (k x-x+y+1)^2} \\ \hline
 (13,12,13) & (-x+y+1)^2 & -\frac{(x-y-1)^2}{(x k-y k-k+x-2 y-1)^2 (x k-y k-k-y-1)} \\ \hline
 (13,13,13) & 1-y & \frac{1-y}{(y k-k+x-2)^2} \\ \hline
 (13,23,13) & x (1-y) & \frac{x-x y}{(k x-y+1) (k x+x-y+1)} \\ \hline
 (13,23,123) & x^2 & \frac{x^2}{(k x-y+1)^2 (k x+x-y+1)} \\ \hline
 (13,123,13) & (-x+y+1)^2 & -\frac{(x-y-1)^2}{(x k-y k-k+x-2) (x k-y k-k+2 x-y-2)^2} \\ \hline
 (13,132,13) & x (1-y) & \frac{x-x y}{(y k-k-x) (y k-k-x+y-1)} \\ \hline
 (23,e,23) & x (-x+y+1) & -\frac{x (x-y-1)}{(k x+y+1) (k x-x+y+1)} \\ \hline
 (23,12,23) & x (-x+y+1) & -\frac{x (x-y-1)}{(x k-y k-k-x) (x k-y k-k-y-1)} \\ \hline
  \end{array}$$
 $$\begin{array}{l|c|c}
 (23,12,132) & (-x+y+1)^2 & -\frac{(x-y-1)^2}{(x k-y k-k-x)^2 (x k-y k-k-y-1)} \\ \hline
 (23,13,23) & (1-y)^2 & -\frac{(y-1)^2}{(y k-k+x-2) (y k-k+x+y-2)^2} \\ \hline
 (23,23,23) & x & \frac{x}{(k x+x-y+1)^2} \\ \hline
 (23,123,23) & (-x+y+1)^2 & -\frac{(x-y-1)^2}{(x k-y k-k+x-2) (x k-y k-k+y-1)^2} \\ \hline
 (23,132,23) & (1-y)^2 & -\frac{(y-1)^2}{(y k-k-x+y-1) (y k-k-x+2 y-1)^2} \\ \hline
 (123,e,132) & x^2 & \frac{x^2}{(k x+y+1) (k x+x+y)^2} \\ \hline
 (123,12,23) & -x+y+1 & \frac{x-y-1}{(x k-y k-k-x) (x k-y k-k-y-1)^2} \\ \hline
 (123,12,132) & -x+y+1 & \frac{x-y-1}{(x k-y k-k-x) (x k-y k-k-y-1)^2} \\ \hline
 (123,13,132) & (1-y) (-x+y+1) & \frac{(x-y-1) (y-1)}{(y k-k+x-2) (y k-k+x-y-1)} \\ \hline
 (123,23,132) & 1-y & \frac{1-y}{(-k x-x+y-1)^2} \\ \hline
 (123,123,23) & (-x+y+1)^2 & -\frac{(x-y-1)^2}{(x k-y k-k+x-2) (x k-y k-k+y-1)^2} \\ \hline
 (123,123,132) & (1-y) (-x+y+1) & \frac{(x-y-1) (y-1)}{(x k-y k-k+x-2) (x k-y k-k+y-1)} \\ \hline
 (123,132,132) & (1-y) (x-y+1) & \frac{(y-1) (-x+y-1)}{(y k-k-x+y-1) (y k-k-x+2 y-2)} \\ \hline
 (132,e,13) & x^2 & \frac{x^2}{(k x+y+1) (k x-x+y+1)^2} \\ \hline
 (132,e,123) & x (-x+y+1) & -\frac{x (x-y-1)}{(k x+y+1) (k x-x+y+1)} \\ \hline
 (132,12,123) & x (-x+y+1) & -\frac{x (x-y-1)}{(x k-y k-k-x) (x k-y k-k-y-1)} \\ \hline
 (132,13,123) & (1-y)^2 & -\frac{(y-1)^2}{(y k-k+x-2) (y k-k+x+y-2)^2} \\ \hline
 (132,23,13) & x^2 & \frac{x^2}{(k x-y+1)^2 (k x+x-y+1)} \\ \hline
 (132,23,123) & x (1-y) & \frac{x-x y}{(k x-y+1) (k x+x-y+1) (k x+2 x-y)} \\ \hline
 (132,123,123) & -x+y+1 & \frac{-x+y+1}{(x k-y k-k+x-2)^2} \\ \hline
 (132,132,123) & (1-y)^2 & -\frac{(y-1)^2}{(y k-k-x+y-1) (y k-k-x+2 y-1)^2}
\end{array}$$

\begin{proof} Linearity is clear.

Let $f(x,y) \in W$.  Then there is real constant $C$ such that  $||g(x,y)| f(x,y)|<C$ for all $(x,y)\in \triangle.$  To show that that $\mathcal{L}_{T_{X}} f(x,y) \in W$ we need to show
$$|g(x,y) \mathcal{L}_{T_{X}} f (x,y)| < C_1$$
for some real constant $C_1.$
We have 
\begin{eqnarray*}
\left|g(x,y)\mathcal{L}_{T_{X}} f(x,y)\right|
&=& |g(x,y)|\left|\sum_{(a,b):T_{X}(a,b)=(x,y)} \frac{1}{\left|\mbox{Jac}(T_{X}(a,b))\right|} f(a,b) \right|\\
&\leq& |g(x,y)|\sum_{(a,b):T_{X}(a,b)=(x,y)}\left|\frac{1}{g(a,b)\left|\mbox{Jac}(T_{X}(a,b))\right|} g(a,b) f(a,b)\right|\\
&\leq& |g(x,y)| C \sum_{(a,b):T_{X}(a,b)=(x,y)}  \left| \frac{1}{g(a,b)\left|\mbox{Jac}(T_{X}(a,b))\right|}\right|.\\
\end{eqnarray*}

Thus we  must show that the sum 
$$\sum_{(a,b):T_{X}(a,b)=(x,y)}  \left| \frac{g(x,y)}{g(a,b)\left|\mbox{Jac}(T_{X}(a,b))\right|}\right|$$
 converges, with a bound independent of $(x,y)$.  In the above table,  each of the summands
$$\left| \frac{g(x,y)}{g(a,b)\left|\mbox{Jac}(T_{X}(a,b))\right|}\right|$$ 
is given. 
By standard arguments it is clear we have our bounds.

\end{proof}

%\begin{center}
%\includegraphics[width=6.5in]{banachnew1.pdf}
%\end{center}
%\bigskip

%\clearpage
%\begin{center}
%\includegraphics[width=6.5in]{banachnew2.pdf}
%\end{center}
%\bigskip

%\begin{center}
%\includegraphics[width=6.5in]{eigbanachnew1.pdf}
%\end{center}
\bigskip

%\clearpage

\subsection{Leading Eigenvalues}
We will show that for $18$ maps the transfer operator has leading eigenvalue one; for four of these maps, the corresponding eigenspace is one-dimensional.  We do not know if analogous results are true for other triangle partition maps, but would be surprised if these $18$ are the only ones with this behavior.

\begin{theorem} The table below lists those triangle partition maps whose transfer operators have an eigenfunction with eigenvalue one and explicitly lists the eigenfunction.
\end{theorem}

$$\begin{array}{c|c}
(\sigma, \tau_0, \tau_1) & \mbox{eigenfunction} \\
\hline
\hline
 (e,e,e) & \frac{1}{x (y+1)} \\\hline
 (e,23,e) & \frac{1}{x (1-y)} \\\hline
 (e,132,e) & \frac{1}{x (1-y)} \\\hline
 (12,12,12) & \frac{1}{(y+1) (-x+y+1)} \\\hline
 (12,13,12) & \frac{1}{(1-y) (-x+y+1)} \\\hline
 (12,123,12) & \frac{1}{(1-y) (-x+y+1)} \\\hline
 (13,13,13) & \frac{1}{(x-2) (1-y)} \\\hline
 (13,23,13) & \frac{1}{x (1-y)} \\\hline
 (13,132,13) & \frac{1}{x (1-y)} \\\hline
 (23,e,23) & \frac{1}{x (-x+y+1)} \\\hline
 (23,12,23) & \frac{1}{x (-x+y+1)} \\\hline
 (23,23,23) & \frac{1}{x (x-y+1)} \\\hline
 (123,13,132) & \frac{1}{(1-y) (-x+y+1)} \\\hline
 (123,123,132) & \frac{1}{(1-y) (-x+y+1)} \\\hline
 (123,132,132) & \frac{1}{(1-y) (x-y+1)} \\\hline
 (132,e,123) & \frac{1}{x (-x+y+1)} \\\hline
 (132,12,123) & \frac{1}{x (-x+y+1)} \\\hline
 (132,123,123) & \frac{1}{(x-2) (-x+y+1)} \\
\end{array}$$

\begin{proof}  These are calculations, in analog to the corresponding calculation in \cite{Garrity15}.
\end{proof}

We now want to show one is the largest eigenvalue for these transfer operators.

For each of the above transfer operators, we have from last section an associated Banach space $W$.  We will need the following theorem.
\begin{theorem}
Fix a triangle partition map and let $W$ be the corresponding Banach space.

For functions $f(x,y), g(x,y) \in W,$ assume that $ f(x,y)< g(x,y)$ for all $(x,y)\in\bigtriangleup.$
Then $$ \mathcal{L}_{T_{\sigma,\tau_0,\tau_1}}^{n}f(x,y)< \mathcal{L}_{T_{\sigma,\tau_0,\tau_1}}^{n} g(x,y)$$
for all $n\geq 0$ for all $(x,y)\in\bigtriangleup$ .

\end{theorem}

\begin{proof}
The proof is a calculation, again in analog to the corresponding calculation in \cite{Garrity15}.

%This follows immediately from the form of the transfer operator $\mathcal{L}_{T_{\sigma, \tau_0,\tau_1}}.$
%Note that by definition, 
%$$\mathcal{L}_{T_{\sigma, \tau_0,\tau_1}}f(x,y) = \sum_{(a,b):T_{\sigma, \tau_0,\tau_1}(a,b)=(x,y)} \frac{1}{\left|\mbox{Jac}%(T_{\sigma, \tau_0,\tau_1}(a,b))\right|} f(a,b).$$
%Clearly, since $\frac{1}{\left|\mbox{Jac}(T_{\sigma, \tau_0,\tau_1}(a,b))\right|}>0$ and by hypothesis $0\leq f(x,y)\leq w(x,y),$ then $$0\leq \mathcal{L}_{T_{\sigma,\tau_0,\tau_1}}^{n}f(x,y)\leq \mathcal{L}_{T_{\sigma,\tau_0,\tau_1}}^{n} w(x,y).$$

\end{proof}

We now want to prove a theorem that will give a general criterion for when one of our transfer operators has leading eigenvalue one. 

For general notation, let $g(x,y)$ be the function that defines the corresponding Banach space $W$ for the triangle partition map $T_{\sigma, \tau_0, \tau_1}$ and suppose that $h(x,y)\in W$ is an eigenvector with eigenvalue one for the corresponding transfer operator.

\begin{theorem}With the above notation, suppose for all $f(x,y) \in W$ that there is a constant $B$ so that for all $(x,y) \in \triangle$ we have 
$$-Bh(x,y) < f(x,y) < B h(x,y).$$
Then the largest eigenvalue of the transfer operator is one.
\end{theorem}

\begin{proof}
Let $\mathcal{L}$ denote the transfer operator and let $\phi(x,y) \in W$ be an eigenfunction with eigenvalue $\lambda$.  By assumption, there is a constant $B$ so that 
$$-Bh(x,y) < \phi(x,y) < B h(x,y).$$
Then for all positive $n$, by applying the transfer operator $n$ times,  we have 
$$-Bh(x,y) < \lambda^n \phi(x,y) < B h(x,y),$$
which means that 
$$-1\leq \lambda \leq 1.$$
\end{proof}

\begin{corr} For the $18$ triangle partition maps in the above table, the corresponding transfer operator has its largest eigenvalue being one.
\end{corr}
\begin{proof} Just explicitly check that each of the elements in the corresponding Banach space satisfies the criterion of the above theorem.
\end{proof}

%\begin{theorem}

%The largest eigenvalue of $\mathcal{L}_{T_{\sigma,\tau_0,\tau_1}}: V \to V$ for every $(\sigma,\tau_0,\tau_1)\in S_{3}^{3}$ for which we have established that $\mathcal{L}_{T_{\sigma,\tau_0,\tau_1}}$ is a linear map from $V$ to $V$ and found a corresponding eigenfunction of eigenvalue 1 is 1.

%\end{theorem}

%\begin{proof}

%Let $f(x,y) \in V$ and suppose that $h(x,y)$ is an eigenfunction of eigenvalue 1 for $\mathcal{L}_{T_{\sigma,\tau_0,\tau_1}}.$ Further suppose that there exist constants $0<a<b$ such that 
%$a h(x,y) \leq f(x,y) \leq b h(x,y).$
%Then by the preceding theorem, for all $n\geq 0,$
%$$\mathcal{L}_{T_{\sigma,\tau_0,\tau_1}}^{n}a h(x,y) \leq \mathcal{L}_{T_{\sigma,\tau_0,\tau_1}}^{n}f(x,y) \leq\mathcal{L}_{T_{\sigma,\tau_0,\tau_1}}^{n} b h(x,y).$$
%Since $h(x,y)$ is an eigenfunction of eigenvalue 1 and $\mathcal{L}_{T_{\sigma,\tau_0,\tau_1}}$ is a linear operator, this implies that
%$$a h(x,y) \leq \mathcal{L}_{T_{\sigma,\tau_0,\tau_1}}^{n}f(x,y) \leq b h(x,y).$$
%Hence, the largest eigenvalue of $\mathcal{L}_{T_{\sigma,\tau_0,\tau_1}}$ must be 1.

%\end{proof}

%\begin{cor}

%By the above theorem, the largest eigenvalue of $\mathcal{L}_{T_{\sigma,\tau_0,\tau_1}}: V \to V$ for every $(\sigma,\tau_0,\tau_1)\in S_{3}^{3}$ listed in statement of Theorem~\ref{theorem:Banach} is 1 since we have established that $\mathcal{L}_{T_{\sigma,\tau_0,\tau_1}}$ is a linear map from $V$ to $V$ and have found a corresponding eigenfunction of eigenvalue 1.

%\end{cor}

\begin{theorem}

The eigenvalue 1 of $\mathcal{L}_{T_{\sigma,\tau_0,\tau_1}}: W \to W$ has multiplicity 1 for every $(\sigma,\tau_0,\tau_1)\in S_{3}$ for which $\mathcal{L}_{T_{\sigma,\tau_0,\tau_1}}$ is an ergodic map from $W$ to $W$ whose eigenfunction of eigenvalue 1 is known.

\end{theorem}

\begin{proof}

The result follows immediately from Theorem 4.2.2 in \cite{Lasota} by Lasota, et al., quoted word-for-word from \cite{Lasota}:

{\it Let $(X,\mathcal{A},\mu)$ be a measure space, $S:X\to X$ a nonsingular transformation, and $P$ the Frobenius-Perron operator associated with $S.$ If $S$ is ergodic, then there is at most one stationary density $f_*$ of $P.$ Further, if there is a unique stationary density $f_*$ of $P$ and $f_*(x)>0$ a.e., then $S$ is ergodic.}

Clearly, for us $S=T_{\sigma,\tau_0,\tau_1},$ $P=\mathcal{L}_{T_{\sigma,\tau_0,\tau_1}},$ and $f_{*}$ is the associated eigenfunction of eigenvalue 1.

\end{proof}

\iffalse
\begin{cor}

The permutation triplets $(\sigma,\tau_0,\tau_1)\in S_{3}^{3}$ for which $T_{\sigma,\tau_0,\tau_1}$ is an ergodic map whose eigenfunction of eigenvalue 1 for the corresponding transfer operator is known are $(e,e,e),$ $(e,23,e),$ $(23,e,23),$ and $(132,12,123).$ By the above theorem, the eigenvalue 1 of each of these transfer operators has multiplicity 1.

\end{cor}
\fi

Unfortunately, ergodicity is not known for most of the TRIP maps.  The TRIP map $T_{e,e,e}$ was shown to be ergodic by Messaoudi,  Nogueira and   Schweiger \cite{Mess}, while $T_{e,23,e},$ $T_{e,23,23},$ $T_{e,132,23},$ $T_{e,23,132}$ were shown to be ergodic by Jensen \cite{Jensen}.  Some preliminary work on ergodicity for some of the other TRIP maps is in \cite{amburg}, which leads us to believe that the above results are also true for quite a few more triangle partition maps.

\section{On a Hilbert Space approach}\label{hilbert}

Mayer and   Roepstorff \cite{Mayer} showed that the transfer operator of the Gauss map is a nuclear operator of trace class zero  when acting on 
$L^2([0,1]).$ In \cite{Garrity15}, it was shown that the transfer operator of the triangle map $T_{e,e,e}$ has a non-trivial analog, namely that the transfer operator should be thought of as an operator acting on a family of Hilbert spaces with respect to the variable $y$, with the variable $x$ treated as a parameter.  This section will show that there are similar results for 44 additional triangle partition maps.  The work is in the same spirit  as in  \cite{Garrity15}, which in turn is in the same spirit, with mirrored proofs, as those given by Mayer and Roepstorff.  The technical details, which are mainly captured  in the functions listed in this section's table, are what makes this a needed distinct paper.

On the non-negative reals, set
$$\mathrm{d}m(t) = \frac{t}{e^t-1} \mathrm{d}t.$$
For each of the triples $X=(\sigma, \tau_0, \tau_1)$ in the below table, we have listed associated functions 
$j(x,y), h(x,y)$ and $l(x,y)$, whose domains are all of  $\triangle.$  
Then set
\begin{eqnarray*}
\eta_k(s) &=& \frac{s^ke^{-s}}{(k+1)!}, \\
e_k(t) &=&L_k^1(t), \\
E_k(x,y) &=&   \int_0^{\infty}j(x,y)e^{-t(l(x,y)-1)}  e_k(t) \mathrm{d}m( t ),\\
\mathcal{K}_{T_{X}}(\phi(x,t)) &=& \int_0^{\infty} \frac{J_1(2\sqrt{st})}{\sqrt{st}} \frac{t}{e^t-1} \phi(x,s) \mathrm{d}m(s), \\
\widehat{\phi} \left(x,y\right) &=& \frac{1}{h(x,y)} \int_{s=0}^{\infty} e^{ -sh(x,y)}\phi   (x,s) \mathrm{d} m(s), \\
\langle \alpha(s), \beta(s) \rangle &=& \int_0^{\infty} \alpha(s)\beta(s) \mathrm{d} m(s).
\end{eqnarray*}
(Here $L_k^1(t) $ denotes the first Laguerre polynomial and  $J_1$ denotes the Bessel function of order one.)
Further, we  have six different transforms, one for each $(\sigma, *,*)$ given by the table

$$\begin{array}{c|c}
(\sigma, *, *) &\mbox{Transform}\; \hat{\phi}(x,y) \\
\hline
\hline
 (e,*,*) &  \frac{1}{h(x,y)} \int_{s=0}^{\infty} e^{-sh(x,y) }\phi\left(  \frac{y}{x}, s  \right) \mathrm{d}m(s) \\\hline
 (12,*,*) &\frac{1}{h(x,y)} \int_{s=0}^{\infty} e^{-sh(x,y) }\phi\left(  \frac{x-1}{y}, s  \right) \mathrm{d}m(s) \\\hline
 (13,*,*) & \frac{1}{h(x,y)} \int_{s=0}^{\infty} e^{-sh(x,y) }\phi\left(  s, \frac{1-x}{1-y}  \right) \mathrm{d}m(s) \\\hline
 (23,*,*) &  \frac{1}{h(x,y)} \int_{s=0}^{\infty} e^{-sh(x,y) }\phi\left(  \frac{y}{x}, s  \right) \mathrm{d}m(s) \\\hline
 (123,*,*) & \frac{1}{h(x,y)} \int_{s=0}^{\infty} e^{-sh(x,y) }\phi\left(  \frac{1-y}{x-y}, s  \right) \mathrm{d}m(s) \\\hline
 (132,*,*) &\frac{1}{h(x,y)} \int_{s=0}^{\infty} e^{-sh(x,y) }\phi\left( s,  \frac{y}{1-x}  \right) \mathrm{d}m(s). \\\hline
\end{array}$$

 \begin{theorem} For functions $f(x,y)$ on $\triangle$ for which there is a function $\phi(x,y)$ such that 
 $f=\hat{\phi}$ and  for which all of the following integrals and sums exist, we have 
 \begin{eqnarray*}
 \mathcal{L}_{T_X}(f(x,y)) &=& j(x,y) \int_0^{\infty} e^{-t(l(x,y)-1)} \mathcal{K}_{T_{X}}(\phi(x,t)) \mathrm{d}t\\
 &=& \sum_{k=0}^{\infty} \langle \phi(x,s), \eta_k(s) \rangle E_k(x,y).
 \end{eqnarray*}
 \end{theorem}
 For each choice from the $44$ triples $(\sigma, \tau_0, \tau_1)$ in the below table we get different allowable functions $f$, which is why in the above theorem we simply state that we want all the integrals to converge.

\begin{proof} In \cite{Garrity15}, it was shown by setting $w=\frac{1+y}{x}$ that the transfer operator $\mathcal{L}_{T_{e,e,e}}$ can be put into the language of Mayer and   Roepstorff \cite{Mayer}, yielding for us the desired results.  For other transfer operators, the change of coordinate is
$$w= l(x,y).$$
Then by direct calculation, it can be shown that 
\begin{eqnarray*}
\mathcal{L}_{T_X}(f(x,y))&=&j(x,y)\int_{0}^{\infty}\sum_{k=0}^{\infty}\frac{1}{(k+w)^2}e^{\frac{-s}{(k+w)}}\phi(\cdot, \cdot\cdot) \mathrm{d} m(s)\\
\end{eqnarray*}
Here each $\phi(\cdot, \cdot\cdot)$ is a function of the variable $s$ and  a term involving just $x$ and $y$. The form of this function depends on the type of transform being used.
The term 
$$\int_{0}^{\infty}\sum_{k=0}^{\infty}\frac{1}{(k+w)^2}e^{\frac{-s}{(k+w)}}\phi(\cdot, \cdot\cdot) \mathrm{d} m(s),$$
treating $\phi(\cdot, \cdot \cdot)$ as a function of $s$, can now be directly studied, independent of the triangle partition map. Here one simply follows the proof of Theorem 1 of Mayer and   Roepstorff \cite{Mayer}.  This has been done in \cite{Garrity15}, allowing us to conclude that the theorem is true.

Again, it is the finding of the appropriate functions $j(x,y), h(x,y)$ and $l(x,y)$ for each $(\sigma, \tau_0, \tau_1)$, given in the below table, is what is important, as once we know these functions, we can systematically translate the calculations to those done in \cite{Garrity15},

\end{proof}
Let us look at the $(123,132,132)$ case as an example.
From the table in Appendix B, we know that the transfer operator 
$\mathcal{L}_{T_{(123,132,132)} } f(x,y)$ is 
 $$ \sum_{k=0}^{\infty}  \frac{1}{(k(1-y) +1 +x-y)^3}    f\left( \frac{k(1-y) +1 +x-2y}{(k(1-y) +1 +x-y)},\frac{k(1-y)  +x-y}{(k(1-y) +1 +x-y)}   \right).$$
 Setting $$w= \frac{1+x-y}{1-y}$$
 we get that the transfer operator is 
$$\frac{1}{(1-y)^3} \sum_{k=0}^{\infty}  \frac{1}{(k+w)^3}    f\left( \frac{k+w - \frac{y}{1-y}}{k+w},\frac{k+w - \frac{1}{1-y}}{k+w}   \right).$$
Set 
$$X=  \frac{k+w - \frac{y}{1-y}}{k+w},\; Y = \frac{k+w - \frac{1}{1-y}}{k+w}.$$
To apply the appropriate transform, we use that 
$$h(X,Y) = X-Y = \frac{1}{k+w},\; \frac{1-Y}{X-Y} = \frac{1}{1-y}. $$
This will give us that the transfer operator is 
$$\frac{1}{(1-y)^3}\int_{0}^{\infty}\sum_{k=0}^{\infty}\frac{1}{(k+w)^2}e^{\frac{-s}{(k+w)}}\phi\left(\frac{1}{1-y}, s \right) \mathrm{d} m(s)$$
as desired.

We do not believe that this theorem is only true for these $44$ triangle partition maps; it is just that for these $44$ maps, we have been able to find the appropriate functions $j(x,y), h(x,y)$ and $l(x,y)$.  The sums for the   other polynomial-type transfer operators all are explicitly written in terms of the parity of the $k$.  By splitting the transfer operator into the sum of the series over even term with the sum of the series over odd terms, we should be able  reproduce what we did above.  These should be capable of being found.

$$\begin{array}{l|c|c|c} \label{Hilbert space}
(\sigma,\tau_0,\tau_1) &   l(x,y) & j(x,y) & h(x,y)\\
\hline
\hline
 (e,e,e) & \frac{y+1}{x} & \frac{1}{x^3} & y \\ \hline
 (e,e,12) & \frac{y+1}{x} & \frac{1}{x^3} & y \\ \hline
 (e,12,e) & \frac{y+1}{-x+y+1} & \frac{1}{(-x+y+1)^3} & y \\ \hline
 (e,13,e) & \frac{x-2}{y-1} & \frac{1}{(1-y)^3} & y \\ \hline
 (e,23,e) & \frac{x-y+1}{x} & \frac{1}{x^3} & y \\ \hline
 (e,23,12) & \frac{x-y+1}{x} & \frac{1}{x^3} & y \\ \hline
 (e,123,e) & \frac{x-2}{x-y-1} & \frac{1}{(-x+y+1)^3} & y \\ \hline
 (e,132,e) & 1-\frac{x}{y-1} & \frac{1}{(1-y)^3} & y \\ \hline
 (12,e,e) & \frac{y+1}{x} & \frac{1}{x^3} & y \\ \hline
 (12,e,12) & \frac{y+1}{x} & \frac{1}{x^3} & y \\ \hline
 (12,12,12) & \frac{y+1}{-x+y+1} & \frac{1}{(-x+y+1)^3} & y \\ \hline
 (12,13,12) & \frac{x-2}{y-1} & \frac{1}{(1-y)^3} & y \\ \hline
 (12,23,e) & \frac{x-y+1}{x} & \frac{1}{x^3} & y \\ \hline
 (12,23,12) & \frac{x-y+1}{x} & \frac{1}{x^3} & y \\ \hline
 (12,123,12) & \frac{x-2}{x-y-1} & \frac{1}{(-x+y+1)^3} & y \\ \hline
 (12,132,12) & 1-\frac{x}{y-1} & \frac{1}{(1-y)^3} & y \\ \hline
 (13,e,13) & \frac{y+1}{x} & \frac{1}{x^3} & 1-x \\ \hline
 (13,e,123) & \frac{y+1}{x} & \frac{1}{x^3} & 1-x \\ \hline
 (13,12,13) & \frac{y+1}{-x+y+1} & \frac{1}{(-x+y+1)^3} & 1-x \\ \hline
 (13,13,13) & \frac{x-2}{y-1} & \frac{1}{(1-y)^3} & 1-x \\ \hline
 (13,23,13) & \frac{x-y+1}{x} & \frac{1}{x^3} & 1-x \\ \hline
 (13,23,123) & \frac{x-y+1}{x} & \frac{1}{x^3} & 1-x
 \end{array}$$
 $$\begin{array}{l|c|c|c} \label{hilbert Space}
 (13,123,13) & \frac{x-2}{x-y-1} & \frac{1}{(-x+y+1)^3} & 1-x \\ \hline
 (13,132,13) & 1-\frac{x}{y-1} & \frac{1}{(1-y)^3} & 1-x \\ \hline
 (23,e,23) & \frac{y+1}{x} & \frac{1}{x^3} & x-y \\ \hline
 (23,12,23) & \frac{y+1}{-x+y+1} & \frac{1}{(-x+y+1)^3} & x-y \\ \hline
 (23,13,23) & \frac{x-2}{y-1} & \frac{1}{(1-y)^3} & x-y \\ \hline
 (23,23,23) & \frac{x-y+1}{x} & \frac{1}{x^3} & x-y \\ \hline
 (23,123,23) & \frac{x-2}{x-y-1} & \frac{1}{(-x+y+1)^3} & x-y \\ \hline
 (23,132,23) & 1-\frac{x}{y-1} & \frac{1}{(1-y)^3} & x-y \\ \hline
 (123,e,132) & \frac{y+1}{x} & \frac{1}{x^3} & x-y \\ \hline
 (123,12,132) & \frac{y+1}{-x+y+1} & \frac{1}{(-x+y+1)^3} & x-y \\ \hline
 (123,13,132) & \frac{x-2}{y-1} & \frac{1}{(1-y)^3} & x-y \\ \hline
 (123,23,132) & \frac{x-y+1}{x} & \frac{1}{x^3} & x-y \\ \hline
 (123,123,132) & \frac{x-2}{x-y-1} & \frac{1}{(-x+y+1)^3} & x-y \\ \hline
 (123,132,132) & 1-\frac{x}{y-1} & \frac{1}{(1-y)^3} & x-y \\ \hline
 (132,e,13) & \frac{y+1}{x} & \frac{1}{x^3} & 1-x \\ \hline
 (132,e,123) & \frac{y+1}{x} & \frac{1}{x^3} & 1-x \\ \hline
 (132,12,123) & \frac{y+1}{-x+y+1} & \frac{1}{(-x+y+1)^3} & 1-x \\ \hline
 (132,13,123) & \frac{x-2}{y-1} & \frac{1}{(1-y)^3} & 1-x \\ \hline
 (132,23,13) & \frac{x-y+1}{x} & \frac{1}{x^3} & 1-x \\ \hline
 (132,23,123) & \frac{x-y+1}{x} & \frac{1}{x^3} & 1-x \\ \hline
 (132,123,123) & \frac{x-2}{x-y-1} & \frac{1}{(-x+y+1)^3} & 1-x \\ \hline
 (132,132,123) & 1-\frac{x}{y-1} & \frac{1}{(1-y)^3} & 1-x 
\end{array}$$

\section{Gauss-Kuzmin Distributions for a Few Triangle Parition Sequences}
\label{GK}
The classical Gauss-Kuzmin statistics give the statistics of the digits occurring in the continued fraction expansion of a number.  It is natural to ask similar questions for any multidimensional continued fraction algorithm (as discussed in Schweiger \cite{Schweiger3}).  In \cite{Garrity15}, the Gauss-Kuzmin distribution was derived for the triangle partition $T_{e,e,e}.$  We will build on that work here.

Fix a triple $(\sigma, \tau_0, \tau_1)$.  Let the $(\sigma, \tau_0, \tau_1)$ expansion for a point $(x,y) \in \triangle$ be $(a_1, a_2, a_3, \ldots)$.   Set 
$$P_{n,k} (x,y) = \frac{\#  \{a_i:a_i=k\; \mbox{and}\; 1\leq i\leq n\}  }{n} .$$
If the limit exists, set
$$P_k(x,y) = \lim_{n\rightarrow \infty} P_{n,k} (x,y) .$$

We want to see when there is a function $p(k)$ that is equal to $P_k(x,y) $ for almost all $(x,y)\in \triangle$ and then find explicit formulas for each $p(k).$

\begin{theorem}
\label{theorem:GK}
Let $(x,y)\in\bigtriangleup$ and suppose $T_{\sigma,\tau_0,\tau_1}$ is a triangle partition  map that is ergodic with respect to the Lebesgue measure $\lambda$ and has associated invariant measure $\mu(a)=\int_{a}r(x,y)$ for $r(x,y)$ as shown in the table in the table below. If $\mu$ and $\lambda$ are absolutely continuous, then $p(k)=\int_{\bigtriangleup_k} \ d\mu$ for almost every $(x,y)\in\bigtriangleup.$
\end{theorem}

$$
\begin{array}{l|c}
 (\sigma,\tau_0,\tau_1) & r(x,y)\\\hline\hline
 (e,e,e) & \frac{12 \text{dxdy}}{\pi ^2 x (y+1)} \\\hline
 (e,23,e) & \frac{6 \text{dxdy}}{\pi ^2 x (1-y)} \\\hline
 (e,132,e) & \frac{6 \text{dxdy}}{\pi ^2 x (1-y)} \\\hline
 (12,12,12) & \frac{12 \text{dxdy}}{\pi ^2 (y+1) (-x+y+1)} \\\hline
 (12,13,12) & \frac{6 \text{dxdy}}{\pi ^2 (1-y) (-x+y+1)} \\\hline
 (12,123,12) & \frac{6 \text{dxdy}}{\pi ^2 (1-y) (-x+y+1)} \\\hline
 (13,13,13) & \frac{12 \text{dxdy}}{\pi ^2 (2-x) (1-y)} \\\hline
 (13,23,13) & \frac{6 \text{dxdy}}{\pi ^2 x (1-y)} \\\hline
 (13,132,13) & \frac{6 \text{dxdy}}{\pi ^2 x (1-y)} \\\hline
 (23,e,23) & \frac{6 \text{dxdy}}{\pi ^2 x (-x+y+1)} \\\hline
 (23,12,23) & \frac{6 \text{dxdy}}{\pi ^2 x (-x+y+1)} \\\hline
 (23,23,23) & \frac{12 \text{dxdy}}{\pi ^2 x (x-y+1)} \\\hline
 (123,13,132) & \frac{6 \text{dxdy}}{\pi ^2 (1-y) (-x+y+1)} \\\hline
 (123,123,132) & \frac{6 \text{dxdy}}{\pi ^2 (1-y) (-x+y+1)} \\\hline
 (123,132,132) & \frac{12 \text{dxdy}}{\pi ^2 (1-y) (x-y+1)} \\\hline
 (132,e,123) & \frac{6 \text{dxdy}}{\pi ^2 x (-x+y+1)} \\\hline
 (132,12,123) & \frac{6 \text{dxdy}}{\pi ^2 x (-x+y+1)} \\\hline
 (132,123,123) & \frac{12 \text{dxdy}}{\pi ^2 (2-x) (-x+y+1)} \\
\end{array}$$

The proof is exactly analogous to that in \cite{Karpenkov}. 

Thus we will have explicit Gauss-Kuzmin statistics if we can calculate $\int_{\bigtriangleup_k} \ d\mu$.  In \cite{Garrity15}, we showed
for  the triangle map $T_{e,e,e}$, that 
$$p(0) =1 -  \frac{6\mathrm{Li}_2\left( \frac{1}{4} \right) + 12\log^2(2) }{\pi^2}$$
and, for $k>0$,
\begin{eqnarray*}p(k)&=& \frac{6}{\pi^2} \left[ \mathrm{Li}_2\left( \frac{1}{(k_1)^2} \right)\right.   - \mathrm{Li}_2\left( \frac{1}{(k+2)^2} \right) \\
&&  + 4\log^2(k+1)- 2\log^2\left( \frac{k+2}{k+1} \right) -2\log(k(k+2))\log(k+1) \bigg]  \end{eqnarray*}

Here, we have an analogous result for $T_{e,23,e}$.

\begin{cor}

For $T_{e,23,e},$  we have 
$$p(0)=\frac{1}{2},$$
 while for $k>0,$

$$p(k)=\int_{\frac{1}{k+2}}^{\frac{1}{k+1}} \left(\int_{\frac{1-x}{k+1}}^x \frac{6}{\pi ^2 x (1-y)} \,
   dy\right) \, dx+\int_{\frac{1}{k+1}}^1 \left(\int_{\frac{1-x}{k+1}}^{\frac{1-x}{k}} \frac{6}{\pi ^2
   x (1-y)} \, dy\right) \, dx.$$

\end{cor}

\begin{proof}

First, we know that $T_{e,23,e}$ is ergodic with respect to the Lebesgue measure. Since for $T_{e,23,e},$ $h(x,y)=\frac{6}{\pi ^2 x (1-y)},$ it follows that $\mu$ and $\lambda$ are absolutely continuous. Hence, by Theorem~\ref{theorem:GK},
it follows that 
$p(k)=\int_{\bigtriangleup_k}d\mu.$
By direct calculation, it is evident that 
$$p(0)=\int_{\bigtriangleup_0}d\mu
=\int_{\frac{1}{2}}^1 \left(\int_{1-x}^x \frac{6}{\pi ^2 x (1-y)} \, dy\right) \, dx
=\frac{1}{2}.$$
while for $k>0,$
$p(k)=\int_{\bigtriangleup_k}d\mu
=\int_{\frac{1}{k+2}}^{\frac{1}{k+1}} \left(\int_{\frac{1-x}{k+1}}^x \frac{6}{\pi ^2 x (1-y)} \,
   dy\right) \, dx+\int_{\frac{1}{k+1}}^1 \left(\int_{\frac{1-x}{k+1}}^{\frac{1-x}{k}} \frac{6}{\pi ^2
   x (1-y)} \, dy\right) \, dx.$

\end{proof}

For the many other triangle partition maps, the Gauss-Kuzmin statistics are still unknown.

\section{Conclusion}\label{conclusion}
This paper and its companion \cite{amburg-Garrity I}   have begun the study of the transfer operators for triangle partition maps, a family  that puts into one framework most well-known multidimensional continued fraction algorithms.  We have seen that for many of these triangle partition maps, basic questions about the transfer operators remain.

We have studied the functional analysis behind triangle partition maps. In particular, we have partitioned this family of 216 multidimensional continued fractions into two natural classes, with exactly half exhibiting what we term ``polynomial behavior" and the other half exhibiting ``non-polynomial behavior."  We  have found eigenfunctions of eigenvalue 1 for transfer operators associated with 17  polynomial-behavior triangle partition maps\footnote{The corresponding eigenfunction for $T_{e,e,e}$ was already well-known.}. The formidably complex form of the non-polynomial-behavior transfer operators and lack of appropriate techniques makes finding leading eigenfunctions prohibitively difficult; however, finding leading eigenfunctions for the remaining 90 polynomial-behavior transfer operators appears a doable, though computationally-intensive, problem.  Proving ergodicity for more triangle partition maps (see \cite{Jensen} and \cite{Mess} for an outline of current ergodicity proofs) would help extend our results to more triangle maps. Finally, we have shown the nuclearity of transfer operators, thought of as acting on one-dimensional families of Hilbert spaces, associated with 44 polynomial-behavior maps. Can transfer operators associated with all polynomial-behavior triangle partition maps be classified as nuclear?

There are also more general questions.  There is the question of connecting the properties of these transfer operators with Diophantine approximation properties (possibly linking with work stemming from \cite{Lagarias93}).  also, as mentioned in \cite{Garrity15}, how much of the rich body of work done over the years on the transfer operator for the Gauss map, pioneered by the work   of Mayer and Roepstorff  \cite{Mayer, Mayer-Roepstorff1988} and of Mayer \cite{Mayer90, Mayer91},  and continuing  today in work of Vall\'{e}e \cite{Vallee97, Vallee98}, of Isola, antoniou and Shkarin \cite{Isola02, antoniou-Shkarin03},  of Jenkinson, Gonzalez and  Urbanksi  \cite{Jenkinson-Gonzalez-Urbanksi03}, of Degli, Espost, Isola and Knauf  \cite{DegliEspost-Isola-Knauf07},  of Hilgert \cite{Hilgert08},  of Bonanno, Graffi and Isola \cite{Bonanno-Graffi-Isola08}, of alkauskas \cite{alkauskas12},  of Iosifescu \cite{Iosifescu14}, of Bonanno and Isola \cite{Bonanno-Isola14},  and of  Ben Ammou, Bonanno, Chouari and Isola \cite{Ben ammou-Bonanno-Chouari-Isola15}, has analogs for triangle partition maps?  These questions strike us as natural but non-trivial.

\clearpage
\appendix

\section{Form of $T_{\sigma,\tau_0,\tau_1}(x,y)$ for Polynomial-Behavior Maps}
\label{app:TRIPmaps}

$$\begin{array}{c|c}
 (e,e,e) & \text{}\left(\frac{y}{x},-\frac{x+k y-1}{x}\right) \\\hline
 (e,e,12) & \text{}\left(\frac{4 y}{(-1)^k (4 x-y-2)-2 k y+y+2},\frac{4-4 k y}{(-1)^k
   (4 x-y-2)-2 k y+y+2}-1\right) \\\hline
 (e,e,23) & \text{}\left(\frac{1}{2} \left(1+(-1)^{k+1}\right)+\frac{(-1)^k
   y}{x},\frac{-\left(2 k+(-1)^k+3\right) x+2 \left(-1+(-1)^k\right) y+4}{4 x}\right) \\\hline
 (e,12,e) & \text{}\left(\frac{1-(k+1) y}{x},-\frac{x+k y-1}{x}\right) \\\hline
 (e,12,12) & \text{}\left(-\frac{4 (k y+y-1)}{(-1)^k (4 x-y-2)-2 k y+y+2},\frac{4-4 k
   y}{(-1)^k (4 x-y-2)-2 k y+y+2}-1\right) \\\hline
 (e,12,23) & \text{}\left(-\frac{2 k x+x+(-1)^{k+1} (x-2 y)+2 y-4}{4 x},\frac{-\left(2
   k+(-1)^k+3\right) x+2 \left(-1+(-1)^k\right) y+4}{4 x}\right) \\\hline
 (e,13,e) & \text{}\left(\frac{2 x+k y-1}{x},1-\frac{y}{x}\right) \\\hline
 (e,13,12) & \text{}\left(\frac{4 k y-4}{(-1)^k (4 x-y-2)-2 k y+y+2}+2,1-\frac{4 y}{(-1)^k (4
   x-y-2)-2 k y+y+2}\right) \\\hline
 (e,13,23) & \text{}\left(\frac{2 k x+7 x+(-1)^k (x-2 y)+2 y-4}{4 x},\frac{x+(-1)^k (x-2 y)}{2
   x}\right) \\\hline
 (e,23,e) & \text{}\left(\frac{y}{x},\frac{x+k y+y-1}{x}\right) \\\hline
 (e,23,12) & \text{}\left(\frac{4 y}{(-1)^k (4 x-y-2)-2 k y+y+2},\frac{4 (k y+y-1)}{(-1)^k (4
   x-y-2)-2 k y+y+2}+1\right) \\\hline
 (e,23,23) & \text{}\left(\frac{1}{2} \left(1+(-1)^{k+1}\right)+\frac{(-1)^k
   y}{x},\frac{\left(2 k+(-1)^{k+1}+5\right) x+2 \left((-1)^k y+y-2\right)}{4 x}\right) \\\hline
 (e,123,e) & \text{}\left(\frac{2 x+k y-1}{x},\frac{x+k y+y-1}{x}\right) \\\hline
 (e,123,12) & \text{}\left(\frac{4 k y-4}{(-1)^k (4 x-y-2)-2 k y+y+2}+2,\frac{4 (k
   y+y-1)}{(-1)^k (4 x-y-2)-2 k y+y+2}+1\right) \\\hline
 (e,123,23) & \text{}\left(\frac{2 k x+7 x+(-1)^k (x-2 y)+2 y-4}{4 x},\frac{\left(2
   k+(-1)^{k+1}+5\right) x+2 \left((-1)^k y+y-2\right)}{4 x}\right) \\\hline
 (e,132,e) & \text{}\left(\frac{1-(k+1) y}{x},1-\frac{y}{x}\right) \\\hline
 (e,132,12) & \text{}\left(-\frac{4 (k y+y-1)}{(-1)^k (4 x-y-2)-2 k y+y+2},1-\frac{4 y}{(-1)^k
   (4 x-y-2)-2 k y+y+2}\right) \\\hline
 (e,132,23) & \text{}\left(-\frac{2 k x+x+(-1)^{k+1} (x-2 y)+2 y-4}{4 x},\frac{x+(-1)^k (x-2
   y)}{2 x}\right) \\
\end{array}$$

$$\begin{array}{c|c}
 (12,e,e) & \text{}\left(\frac{4 y}{(-1)^{k+1} (4 x-3 y-2)-2 k y+y+2},\frac{4 k y-4}{(-1)^k (4 x-3 y-2)+2 k y-y-2}-1\right) \\\hline
 (12,e,12) & \text{}\left(\frac{y}{-x+y+1},\frac{-x+k y+y}{x-y-1}\right) \\\hline
 (12,e,123) & \text{}\left(\frac{1}{2}-\frac{(-1)^k (x+y-1)}{2 (x-y-1)},\frac{-3 x+5 y+2 k (-x+y+1)+(-1)^{k+1} (x+y-1)-1}{4 (x-y-1)}\right) \\\hline
 (12,12,e) & \text{}\left(\frac{4 (k y+y-1)}{(-1)^k (4 x-3 y-2)+2 k y-y-2},\frac{4 k y-4}{(-1)^k (4 x-3 y-2)+2 k y-y-2}-1\right) \\\hline
 (12,12,12) & \text{}\left(\frac{k y+y-1}{x-y-1},\frac{-x+k y+y}{x-y-1}\right) \\\hline
 (12,12,123) & \text{}\left(\frac{-x+3 y+2 k (-x+y+1)+(-1)^k (x+y-1)-3}{4 (x-y-1)},\frac{-3 x+5 y+2 k (-x+y+1)+(-1)^{k+1} (x+y-1)-1}{4 (x-y-1)}\right) \\\hline
 (12,13,e) & \text{}\left(\frac{4-4 k y}{(-1)^k (4 x-3 y-2)+2 k y-y-2}+2,1-\frac{4 y}{(-1)^{k+1} (4 x-3 y-2)-2 k y+y+2}\right) \\\hline
 (12,13,12) & \text{}\left(\frac{-2 x+(k+2) y+1}{-x+y+1},\frac{x-1}{x-y-1}\right) \\\hline
 (12,13,123) & \text{}\left(\frac{7 x+2 k (x-y-1)-9 y+(-1)^k (x+y-1)-3}{4 (x-y-1)},\frac{1}{2} \left(\frac{(-1)^k (x+y-1)}{x-y-1}+1\right)\right) \\\hline
 (12,23,e) & \text{}\left(\frac{4 y}{(-1)^{k+1} (4 x-3 y-2)-2 k y+y+2},\frac{4-4 (k+1) y}{(-1)^k (4 x-3 y-2)+2 k y-y-2}+1\right) \\\hline
 (12,23,12) & \text{}\left(\frac{y}{-x+y+1},\frac{(k+2) y-x}{-x+y+1}\right) \\\hline
 (12,23,123) & \text{}\left(\frac{1}{2}-\frac{(-1)^k (x+y-1)}{2 (x-y-1)},\frac{5 x+2 k (x-y-1)-7 y+(-1)^{k+1} (x+y-1)-1}{4 (x-y-1)}\right) \\\hline
 (12,123,e) & \text{}\left(\frac{4-4 k y}{(-1)^k (4 x-3 y-2)+2 k y-y-2}+2,\frac{4-4 (k+1) y}{(-1)^k (4 x-3 y-2)+2 k y-y-2}+1\right) \\\hline
 (12,123,12) & \text{}\left(\frac{-2 x+(k+2) y+1}{-x+y+1},\frac{(k+2) y-x}{-x+y+1}\right) \\\hline
 (12,123,123) & \text{}\left(\frac{7 x+2 k (x-y-1)-9 y+(-1)^k (x+y-1)-3}{4 (x-y-1)},\frac{5 x+2 k (x-y-1)-7 y+(-1)^{k+1} (x+y-1)-1}{4 (x-y-1)}\right) \\\hline
 (12,132,e) & \text{}\left(\frac{4 (k y+y-1)}{(-1)^k (4 x-3 y-2)+2 k y-y-2},1-\frac{4 y}{(-1)^{k+1} (4 x-3 y-2)-2 k y+y+2}\right) \\\hline
 (12,132,12) & \text{}\left(\frac{k y+y-1}{x-y-1},\frac{x-1}{x-y-1}\right) \\\hline
 (12,132,123) & \text{}\left(\frac{-x+3 y+2 k (-x+y+1)+(-1)^k (x+y-1)-3}{4 (x-y-1)},\frac{1}{2} \left(\frac{(-1)^k (x+y-1)}{x-y-1}+1\right)\right) \\
\end{array}$$

$$\begin{array}{c|c}
 (13,e,13) & \text{}\left(\frac{x-1}{y-1},\frac{-x k+k-y}{y-1}\right) \\\hline
 (13,e,123) & \text{}\left(\frac{4-4 x}{2 k (x-1)-x+(-1)^k (x-4 y+1)+3},\frac{4 k (x-1)+4}{2 k (x-1)-x+(-1)^k (x-4 y+1)+3}-1\right) \\\hline
 (13,e,132) & \text{}\left(\frac{1}{2}-\frac{(-1)^k (-2 x+y+1)}{2 (y-1)},\frac{-2 x+(-1)^k (2 x-y-1)-2 k (y-1)-3 y+1}{4 (y-1)}\right) \\\hline
 (13,12,13) & \text{}\left(\frac{k-(k+1) x}{y-1},\frac{-x k+k-y}{y-1}\right) \\\hline
 (13,12,123) & \text{}\left(\frac{4 (k (x-1)+x)}{2 k (x-1)-x+(-1)^k (x-4 y+1)+3},\frac{4 k (x-1)+4}{2 k (x-1)-x+(-1)^k (x-4 y+1)+3}-1\right) \\\hline
 (13,12,132) & \text{}\left(\frac{-2 \left(1+(-1)^k\right) x-2 k (y-1)+\left(-1+(-1)^k\right) (y+1)}{4 (y-1)},\frac{-2 x+(-1)^k (2 x-y-1)-2 k (y-1)-3 y+1}{4 (y-1)}\right) \\\hline
 (13,13,13) & \text{}\left(\frac{k (x-1)+1}{y-1}+2,\frac{y-x}{y-1}\right) \\\hline
 (13,13,123) & \text{}\left(\frac{1}{\frac{k (x-1)+1}{-x+(-1)^k (x-4 y+1)+1}+\frac{1}{2}},\frac{4 (x-1)}{2 k (x-1)-x+(-1)^k (x-4 y+1)+3}+1\right) \\\hline
 (13,13,132) & \text{}\left(\frac{-2 \left(-1+(-1)^k\right) x+2 k (y-1)+7 y+(-1)^k (y+1)-5}{4 (y-1)},\frac{1}{2} \left(\frac{(-1)^k (-2 x+y+1)}{y-1}+1\right)\right) \\\hline
 (13,23,13) & \text{}\left(\frac{x-1}{y-1},\frac{(k+1) (x-1)+y}{y-1}\right) \\\hline
 (13,23,123) & \text{}\left(\frac{4-4 x}{2 k (x-1)-x+(-1)^k (x-4 y+1)+3},\frac{4 k-4 (k+1) x}{2 k (x-1)-x+(-1)^k (x-4 y+1)+3}+1\right) \\\hline
 (13,23,132) & \text{}\left(\frac{1}{2}-\frac{(-1)^k (-2 x+y+1)}{2 (y-1)},\frac{2 x+(-1)^k (2 x-y-1)+2 k (y-1)+5 y-3}{4 (y-1)}\right) \\\hline
 (13,123,13) & \text{}\left(\frac{k (x-1)+1}{y-1}+2,\frac{(k+1) (x-1)+y}{y-1}\right) \\\hline
 (13,123,123) & \text{}\left(\frac{1}{\frac{k (x-1)+1}{-x+(-1)^k (x-4 y+1)+1}+\frac{1}{2}},\frac{4 k-4 (k+1) x}{2 k (x-1)-x+(-1)^k (x-4 y+1)+3}+1\right) \\\hline
 (13,123,132) & \text{}\left(\frac{-2 \left(-1+(-1)^k\right) x+2 k (y-1)+7 y+(-1)^k (y+1)-5}{4 (y-1)},\frac{2 x+(-1)^k (2 x-y-1)+2 k (y-1)+5 y-3}{4 (y-1)}\right) \\\hline
 (13,132,13) & \text{}\left(\frac{k-(k+1) x}{y-1},\frac{y-x}{y-1}\right) \\\hline
 (13,132,123) & \text{}\left(\frac{4 (k (x-1)+x)}{2 k (x-1)-x+(-1)^k (x-4 y+1)+3},\frac{4 (x-1)}{2 k (x-1)-x+(-1)^k (x-4 y+1)+3}+1\right) \\\hline
 (13,132,132) & \text{}\left(\frac{-2 \left(1+(-1)^k\right) x-2 k (y-1)+\left(-1+(-1)^k\right) (y+1)}{4 (y-1)},\frac{1}{2} \left(\frac{(-1)^k (-2 x+y+1)}{y-1}+1\right)\right) \\
\end{array}$$

$$\begin{array}{c|c}
 (23,e,e) & \text{}\left(\frac{x+(-1)^k (x-2 y)}{2 x},\frac{-2 k x-5 x+(-1)^k (x-2 y)+2 y+4}{4 x}\right) \\\hline
 (23,e,23) & \text{}\left(1-\frac{y}{x},\frac{-(k+1) x+k y+1}{x}\right) \\\hline
 (23,e,132) & \text{}\left(\frac{4 (x-y)}{-2 k x+x+2 k y-y+(-1)^k (3 x+y-2)+2},\frac{4 k (y-x)+4}{-2 k x+x+2 k y-y+(-1)^k (3 x+y-2)+2}-1\right) \\\hline
 (23,12,e) & \text{}\left(\frac{2 \left((-1)^k y+y+2\right)-\left(2 k+(-1)^k+3\right) x}{4 x},\frac{-2 k x-5 x+(-1)^k (x-2 y)+2 y+4}{4 x}\right) \\\hline
 (23,12,23) & \text{}\left(\frac{-(k+1) x+k y+y+1}{x},\frac{-(k+1) x+k y+1}{x}\right) \\\hline
 (23,12,132) & \text{}\left(\frac{4 (-(k+1) x+k y+y+1)}{-2 k x+x+2 k y-y+(-1)^k (3 x+y-2)+2},\frac{4 k (y-x)+4}{-2 k x+x+2 k y-y+(-1)^k (3 x+y-2)+2}-1\right) \\\hline
 (23,13,e) & \text{}\left(\frac{\left(2 k+(-1)^{k+1}+9\right) x+2 \left(-1+(-1)^k\right) y-4}{4 x},\frac{1}{2} \left(1+(-1)^{k+1}\right)+\frac{(-1)^k y}{x}\right) \\\hline
 (23,13,23) & \text{}\left(\frac{(k+2) x-k y-1}{x},\frac{y}{x}\right) \\\hline
 (23,13,132) & \text{}\left(\frac{1}{\frac{k (y-x)+1}{x-y+(-1)^k (3 x+y-2)}+\frac{1}{2}},\frac{4 y-4 x}{-2 k x+x+2 k y-y+(-1)^k (3 x+y-2)+2}+1\right) \\\hline
 (23,23,e) & \text{}\left(\frac{x+(-1)^k (x-2 y)}{2 x},\frac{\left(2 k+(-1)^k+7\right) x-2 \left((-1)^k y+y+2\right)}{4 x}\right) \\\hline
 (23,23,23) & \text{}\left(1-\frac{y}{x},\frac{(k+2) x-(k+1) y-1}{x}\right) \\\hline
 (23,23,132) & \text{}\left(\frac{4 (x-y)}{-2 k x+x+2 k y-y+(-1)^k (3 x+y-2)+2},\frac{4 (k x+x-(k+1) y-1)}{-2 k x+x+2 k y-y+(-1)^k (3 x+y-2)+2}+1\right) \\\hline
 (23,123,e) & \text{}\left(\frac{\left(2 k+(-1)^{k+1}+9\right) x+2 \left(-1+(-1)^k\right) y-4}{4 x},\frac{\left(2 k+(-1)^k+7\right) x-2 \left((-1)^k y+y+2\right)}{4 x}\right) \\\hline
 (23,123,23) & \text{}\left(\frac{(k+2) x-k y-1}{x},\frac{(k+2) x-(k+1) y-1}{x}\right) \\\hline
 (23,123,132) & \text{}\left(\frac{1}{\frac{k (y-x)+1}{x-y+(-1)^k (3 x+y-2)}+\frac{1}{2}},\frac{4 (k x+x-(k+1) y-1)}{-2 k x+x+2 k y-y+(-1)^k (3 x+y-2)+2}+1\right) \\\hline
 (23,132,e) & \text{}\left(\frac{2 \left((-1)^k y+y+2\right)-\left(2 k+(-1)^k+3\right) x}{4 x},\frac{1}{2} \left(1+(-1)^{k+1}\right)+\frac{(-1)^k y}{x}\right) \\\hline
 (23,132,23) & \text{}\left(\frac{-(k+1) x+k y+y+1}{x},\frac{y}{x}\right) \\\hline
 (23,132,132) & \text{}\left(\frac{4 (-(k+1) x+k y+y+1)}{-2 k x+x+2 k y-y+(-1)^k (3 x+y-2)+2},\frac{4 y-4 x}{-2 k x+x+2 k y-y+(-1)^k (3 x+y-2)+2}+1\right) \\
\end{array}$$

$$\begin{array}{c|c}
 (123,e,13) & \text{}\left(\frac{1}{2} \left(\frac{(-1)^k (-2 x+y+1)}{y-1}+1\right),-\frac{-2 x+(-1)^k (2 x-y-1)+2 k (y-1)+5 y+1}{4 (y-1)}\right) \\\hline
 (123,e,23) & \text{}\left(\frac{4 y-4 x}{2 k x-x-2 k y+y+(-1)^k (x+3 y-2)-2},\frac{4 k (x-y)-4}{2 k x-x-2 k y+y+(-1)^k (x+3 y-2)-2}-1\right) \\\hline
 (123,e,132) & \text{}\left(\frac{y-x}{y-1},\frac{k x-(k+1) y}{y-1}\right) \\\hline
 (123,12,13) & \text{}\left(-\frac{-2 \left(1+(-1)^k\right) x+2 k (y-1)+\left(3+(-1)^k\right) (y+1)}{4 (y-1)},-\frac{-2 x+(-1)^k (2 x-y-1)+2 k (y-1)+5 y+1}{4 (y-1)}\right) \\\hline
 (123,12,23) & \text{}\left(\frac{4 (k x+x-(k+1) y-1)}{2 k x-x-2 k y+y+(-1)^k (x+3 y-2)-2},\frac{4 k (x-y)-4}{2 k x-x-2 k y+y+(-1)^k (x+3 y-2)-2}-1\right) \\\hline
 (123,12,132) & \text{}\left(\frac{k x+x-(k+1) y-1}{y-1},\frac{k x-(k+1) y}{y-1}\right) \\\hline
 (123,13,13) & \text{}\left(\frac{-2 x+(-1)^k (2 x-y-1)+2 k (y-1)+9 y-3}{4 (y-1)},\frac{1}{2}-\frac{(-1)^k (-2 x+y+1)}{2 (y-1)}\right) \\\hline
 (123,13,23) & \text{}\left(\frac{1}{\frac{k (x-y)-1}{-x+y+(-1)^k (x+3 y-2)}+\frac{1}{2}},\frac{4 (x-y)}{2 k x-x-2 k y+y+(-1)^k (x+3 y-2)-2}+1\right) \\\hline
 (123,13,132) & \text{}\left(k+\frac{-x k+k+1}{y-1}+2,\frac{x-1}{y-1}\right) \\\hline
 (123,23,13) & \text{}\left(\frac{1}{2} \left(\frac{(-1)^k (-2 x+y+1)}{y-1}+1\right),\frac{-2 \left(1+(-1)^k\right) x+2 k (y-1)+7 y+(-1)^k (y+1)-1}{4 (y-1)}\right) \\\hline
 (123,23,23) & \text{}\left(\frac{4 y-4 x}{2 k x-x-2 k y+y+(-1)^k (x+3 y-2)-2},\frac{4 (-(k+1) x+k y+y+1)}{2 k x-x-2 k y+y+(-1)^k (x+3 y-2)-2}+1\right) \\\hline
 (123,23,132) & \text{}\left(\frac{y-x}{y-1},\frac{(k+2) y-(k+1) x}{y-1}\right) \\\hline
 (123,123,13) & \text{}\left(\frac{-2 x+(-1)^k (2 x-y-1)+2 k (y-1)+9 y-3}{4 (y-1)},\frac{-2 \left(1+(-1)^k\right) x+2 k (y-1)+7 y+(-1)^k (y+1)-1}{4 (y-1)}\right) \\\hline
 (123,123,23) & \text{}\left(\frac{1}{\frac{k (x-y)-1}{-x+y+(-1)^k (x+3 y-2)}+\frac{1}{2}},\frac{4 (-(k+1) x+k y+y+1)}{2 k x-x-2 k y+y+(-1)^k (x+3 y-2)-2}+1\right) \\\hline
 (123,123,132) & \text{}\left(k+\frac{-x k+k+1}{y-1}+2,\frac{(k+2) y-(k+1) x}{y-1}\right) \\\hline
 (123,132,13) & \text{}\left(-\frac{-2 \left(1+(-1)^k\right) x+2 k (y-1)+\left(3+(-1)^k\right) (y+1)}{4 (y-1)},\frac{1}{2}-\frac{(-1)^k (-2 x+y+1)}{2 (y-1)}\right) \\\hline
 (123,132,23) & \text{}\left(\frac{4 (k x+x-(k+1) y-1)}{2 k x-x-2 k y+y+(-1)^k (x+3 y-2)-2},\frac{4 (x-y)}{2 k x-x-2 k y+y+(-1)^k (x+3 y-2)-2}+1\right) \\\hline
 (123,132,132) & \text{}\left(\frac{k x+x-(k+1) y-1}{y-1},\frac{x-1}{y-1}\right) \\
\end{array}$$

$$\begin{array}{c|c}
 (132,e,12) & \text{}\left(\frac{1}{2} \left(\frac{(-1)^k (x+y-1)}{x-y-1}+1\right),\frac{-5 x+3 y+2 k (-x+y+1)+(-1)^k (x+y-1)+1}{4 (x-y-1)}\right) \\\hline
 (132,e,13) & \text{}\left(\frac{4 (x-1)}{-2 k (x-1)+x+(-1)^k (3 x-4 y-1)-3},\frac{4 k (x-1)+4}{2 k (x-1)-x+(-1)^{k+1} (3 x-4 y-1)+3}-1\right) \\\hline
 (132,e,123) & \text{}\left(\frac{x-1}{x-y-1},\frac{k (x-1)+1}{-x+y+1}-1\right) \\\hline
 (132,12,12) & \text{}\left(\frac{-3 x+y+2 k (-x+y+1)+(-1)^{k+1} (x+y-1)-1}{4 (x-y-1)},\frac{-5 x+3 y+2 k (-x+y+1)+(-1)^k (x+y-1)+1}{4 (x-y-1)}\right) \\\hline
 (132,12,13) & \text{}\left(\frac{4 (k (x-1)+x)}{2 k (x-1)-x+(-1)^{k+1} (3 x-4 y-1)+3},\frac{4 k (x-1)+4}{2 k (x-1)-x+(-1)^{k+1} (3 x-4 y-1)+3}-1\right) \\\hline
 (132,12,123) & \text{}\left(\frac{k (x-1)+x}{-x+y+1},\frac{k (x-1)+1}{-x+y+1}-1\right) \\\hline
 (132,13,12) & \text{}\left(\frac{9 x+2 k (x-y-1)-7 y+(-1)^{k+1} (x+y-1)-5}{4 (x-y-1)},\frac{1}{2}-\frac{(-1)^k (x+y-1)}{2 (x-y-1)}\right) \\\hline
 (132,13,13) & \text{}\left(\frac{-4 k (x-1)-4}{2 k (x-1)-x+(-1)^{k+1} (3 x-4 y-1)+3}+2,\frac{4-4 x}{-2 k (x-1)+x+(-1)^k (3 x-4 y-1)-3}+1\right) \\\hline
 (132,13,123) & \text{}\left(\frac{-x k+k-1}{-x+y+1}+2,\frac{y}{-x+y+1}\right) \\\hline
 (132,23,12) & \text{}\left(\frac{1}{2} \left(\frac{(-1)^k (x+y-1)}{x-y-1}+1\right),\frac{7 x+2 k (x-y-1)-5 y+(-1)^k (x+y-1)-3}{4 (x-y-1)}\right) \\\hline
 (132,23,13) & \text{}\left(\frac{4 (x-1)}{-2 k (x-1)+x+(-1)^k (3 x-4 y-1)-3},\frac{4 k-4 (k+1) x}{2 k (x-1)-x+(-1)^{k+1} (3 x-4 y-1)+3}+1\right) \\\hline
 (132,23,123) & \text{}\left(\frac{x-1}{x-y-1},\frac{-x k+k-2 x+y+1}{-x+y+1}\right) \\\hline
 (132,123,12) & \text{}\left(\frac{9 x+2 k (x-y-1)-7 y+(-1)^{k+1} (x+y-1)-5}{4 (x-y-1)},\frac{7 x+2 k (x-y-1)-5 y+(-1)^k (x+y-1)-3}{4 (x-y-1)}\right) \\\hline
 (132,123,13) & \text{}\left(\frac{-4 k (x-1)-4}{2 k (x-1)-x+(-1)^{k+1} (3 x-4 y-1)+3}+2,\frac{4 k-4 (k+1) x}{2 k (x-1)-x+(-1)^{k+1} (3 x-4 y-1)+3}+1\right) \\\hline
 (132,123,123) & \text{}\left(\frac{-x k+k-1}{-x+y+1}+2,\frac{-x k+k-2 x+y+1}{-x+y+1}\right) \\\hline
 (132,132,12) & \text{}\left(\frac{-3 x+y+2 k (-x+y+1)+(-1)^{k+1} (x+y-1)-1}{4 (x-y-1)},\frac{1}{2}-\frac{(-1)^k (x+y-1)}{2 (x-y-1)}\right) \\\hline
 (132,132,13) & \text{}\left(\frac{4 (k (x-1)+x)}{2 k (x-1)-x+(-1)^{k+1} (3 x-4 y-1)+3},\frac{4-4 x}{-2 k (x-1)+x+(-1)^k (3 x-4 y-1)-3}+1\right) \\\hline
 (132,132,123) & \text{}\left(\frac{k (x-1)+x}{-x+y+1},\frac{y}{-x+y+1}\right) \\
\end{array}$$

\section{Form of $\mathcal{L}_{T_{\sigma,\tau_0,\tau_1}}f(x,y)$ for Polynomial-Behavior Maps}
\label{app:formoftransfer}

$$\scalebox{0.9}{$\begin{array}{c|c}
 (e,e,e) & \sum_{k=0}^{\infty}\frac{1}{(k x+y+1)^3}f\left(\frac{1}{k x+y+1},\frac{x}{k x+y+1}\right) \\\hline
 (e,e,12) & \sum_{k=0}^{\infty}\frac{1}{(k x+y+1)^3}f\left(\frac{(-1)^k \left(-x-2 y+(-1)^k (2 k x+x+2 y+2)+2\right)}{4 (k x+y+1)},\frac{x}{k x+y+1}\right) \\\hline
 (e,e,23) & \sum_{k=0}^{\infty}\frac{64}{\left| 2 x+(-1)^k (2 k-2 x+4 y+5)-1\right| ^3}f\left(\frac{4 (-1)^k}{2 x+(-1)^k (2 k-2 x+4 y+5)-1},-\frac{2 \left(2 x+(-1)^k-1\right)}{-2 x+(-1)^{k+1} (2 k-2 x+4
   y+5)+1}\right) \\\hline
 (e,12,e) & \sum_{k=0}^{\infty}\frac{1}{(-x k+y k+k+y+1)^3}f\left(\frac{1}{-x k+y k+k+y+1},\frac{-x+y+1}{-x k+y k+k+y+1}\right) \\\hline
 (e,12,12) & \sum_{k=0}^{\infty}\frac{1}{(-x k+y k+k+y+1)^3}f\left(-\frac{x+(-1)^{k+1} (x-3 y+1)+2 k (x-y-1)-3 y-3}{4 (-x k+y k+k+y+1)},\frac{-x+y+1}{-x k+y k+k+y+1}\right) \\\hline
 (e,12,23) & \sum_{k=0}^{\infty}\frac{64}{\left| -2 x+2 y+(-1)^k (2 k+2 x+2 y+3)+1\right| ^3}f\left(\frac{4 (-1)^k}{-2 x+2 y+(-1)^k (2 k+2 x+2 y+3)+1},\frac{2 \left(-2 x+(-1)^k+2 y+1\right)}{-2 x+2 y+(-1)^k (2 k+2 x+2
   y+3)+1}\right) \\\hline
 (e,13,e) & \sum_{k=0}^{\infty}\frac{1}{(-y k+k-x+2)^3}f\left(\frac{1}{-y k+k-x+2},\frac{y-1}{x+k (y-1)-2}\right) \\\hline
 (e,13,12) & \sum_{k=0}^{\infty}\frac{1}{\left| x+k (y-1)-2\right| ^3}f\left(\frac{(-1)^k \left(-2 x-y+(-1)^k (2 x+2 k (y-1)+y-5)+1\right)}{4 (x+k (y-1)-2)},\frac{y-1}{x+k (y-1)-2}\right) \\\hline
 (e,13,23) & \sum_{k=0}^{\infty}\frac{64}{\left| -2 y+(-1)^k (2 k-4 x+2 y+7)+1\right| ^3}f\left(\frac{4 (-1)^k}{-2 y+(-1)^k (2 k-4 x+2 y+7)+1},\frac{2 \left(-2 y+(-1)^k+1\right)}{-2 y+(-1)^k (2 k-4 x+2 y+7)+1}\right) \\\hline
 (e,23,e) & \sum_{k=0}^{\infty}\frac{1}{(k x+x-y+1)^3}f\left(\frac{1}{k x+x-y+1},\frac{x}{k x+x-y+1}\right) \\\hline
 (e,23,12) & \sum_{k=0}^{\infty}\frac{1}{(k x+x-y+1)^3}f\left(\frac{(-1)^k \left(-3 x+(-1)^k ((2 k+3) x-2 y+2)+2 y+2\right)}{4 (k x+x-y+1)},\frac{x}{k x+x-y+1}\right) \\\hline
 (e,23,23) & \sum_{k=0}^{\infty}\frac{64}{\left| 2 x+(-1)^k (2 k+2 x-4 y+5)-1\right| ^3}f\left(\frac{4 (-1)^k}{2 x+(-1)^k (2 k+2 x-4 y+5)-1},\frac{2 \left(2 x+(-1)^k-1\right)}{2 x+(-1)^k (2 k+2 x-4 y+5)-1}\right) \\\hline
 (e,123,e) & \sum_{k=0}^{\infty}\frac{1}{(-x+k (-x+y+1)+2)^3}f\left(\frac{1}{-x k+y k+k-x+2},\frac{-x+y+1}{-x k+y k+k-x+2}\right) \\\hline
 (e,123,12) & \sum_{k=0}^{\infty}\frac{1}{\left| x+k (x-y-1)-2\right| ^3}f\left(\frac{(-1)^k \left(-3 x+(-1)^k (3 x+2 k (x-y-1)-y-5)+y+1\right)}{4 (x+k (x-y-1)-2)},\frac{-x+y+1}{-x k+y k+k-x+2}\right) \\\hline
 (e,123,23) & \sum_{k=0}^{\infty}\frac{64}{\left| -2 x+(-1)^k (2 k-2 x-2 y+7)+2 y+1\right| ^3}f\left(\frac{4 (-1)^k}{-2 x+(-1)^k (2 k-2 x-2 y+7)+2 y+1},-\frac{2 \left(-2 x+(-1)^k+2 y+1\right)}{2 x-2 y+(-1)^k (-2 k+2 x+2
   y-7)-1}\right) \\\hline
 (e,132,e) & \sum_{k=0}^{\infty}\frac{1}{(k+x-(k+1) y+1)^3}f\left(\frac{1}{k+x-(k+1) y+1},\frac{y-1}{-x+k (y-1)+y-1}\right) \\\hline
 (e,132,12) & \sum_{k=0}^{\infty}\frac{1}{(k+x-(k+1) y+1)^3}f\left(\frac{(-1)^k \left(2 x-3 y+(-1)^k (-2 x+2 k (y-1)+3 y-3)-1\right)}{4 (k+1) y-4 (k+x+1)},\frac{y-1}{-x+k (y-1)+y-1}\right) \\\hline
 (e,132,23) & \sum_{k=0}^{\infty}\frac{64}{\left| (-1)^k (2 k+4 x-2 y+3)-2 y+1\right| ^3}f\left(\frac{4 (-1)^k}{(-1)^k (2 k+4 x-2 y+3)-2 y+1},-\frac{2 \left(-2 y+(-1)^k+1\right)}{(-1)^{k+1} (2 k+4 x-2 y+3)+2 y-1}\right)
   \\
\end{array}$}$$

$$\scalebox{0.9}{$\begin{array}{c|c}
 (12,e,e) & \sum_{k=0}^{\infty}\frac{1}{(k x+y+1)^3}f\left(\frac{(-1)^k \left(x+2 y+(-1)^k (2 k x+3 x+2 y+2)-2\right)}{4 (k x+y+1)},\frac{x}{k x+y+1}\right) \\\hline
 (12,e,12) & \sum_{k=0}^{\infty}\frac{1}{(k x+y+1)^3}f\left(\frac{k x+x+y}{k x+y+1},\frac{x}{k x+y+1}\right) \\\hline
 (12,e,123) & \sum_{k=0}^{\infty}\frac{64}{\left| 2 x+(-1)^k (2 k-2 x+4 y+5)-1\right| ^3}f\left(\frac{6 x+(-1)^k (2 k-2 x+4 y+3)-3}{2 x+(-1)^k (2 k-2 x+4 y+5)-1},-\frac{2 \left(2 x+(-1)^k-1\right)}{-2 x+(-1)^{k+1} (2 k-2
   x+4 y+5)+1}\right) \\\hline
 (12,12,e) & \sum_{k=0}^{\infty}\frac{1}{(-x k+y k+k+y+1)^3}f\left(-\frac{(-1)^k \left(x-3 y+(-1)^k (3 x+2 k (x-y-1)-5 (y+1))+1\right)}{4 (y+k (-x+y+1)+1)},\frac{-x+y+1}{-x k+y k+k+y+1}\right) \\\hline
 (12,12,12) & \sum_{k=0}^{\infty}\frac{1}{(-x k+y k+k+y+1)^3}f\left(\frac{k-(k+1) x+(k+2) y+1}{y+k (-x+y+1)+1},\frac{-x+y+1}{-x k+y k+k+y+1}\right) \\\hline
 (12,12,123) & \sum_{k=0}^{\infty}\frac{64}{\left| -2 x+2 y+(-1)^k (2 k+2 x+2 y+3)+1\right| ^3}f\left(\frac{-6 x+6 y+(-1)^k (2 k+2 x+2 y+1)+3}{-2 x+2 y+(-1)^k (2 k+2 x+2 y+3)+1},\frac{2 \left(-2 x+(-1)^k+2 y+1\right)}{-2 x+2
   y+(-1)^k (2 k+2 x+2 y+3)+1}\right) \\\hline
 (12,13,e) & \sum_{k=0}^{\infty}\frac{1}{\left| x+k (y-1)-2\right| ^3}f\left(\frac{(-1)^k \left(2 x+y+(-1)^k (2 x+2 k (y-1)+3 y-7)-1\right)}{4 (x+k (y-1)-2)},\frac{y-1}{x+k (y-1)-2}\right) \\\hline
 (12,13,12) & \sum_{k=0}^{\infty}-\frac{1}{(x+k (y-1)-2)^3}f\left(\frac{x+k (y-1)+y-2}{x+k (y-1)-2},\frac{y-1}{x+k (y-1)-2}\right) \\\hline
 (12,13,123) & \sum_{k=0}^{\infty}\frac{64}{\left| -2 y+(-1)^k (2 k-4 x+2 y+7)+1\right| ^3}f\left(\frac{-6 y+(-1)^k (2 k-4 x+2 y+5)+3}{-2 y+(-1)^k (2 k-4 x+2 y+7)+1},\frac{2 \left(-2 y+(-1)^k+1\right)}{-2 y+(-1)^k (2 k-4 x+2
   y+7)+1}\right) \\\hline
 (12,23,e) & \sum_{k=0}^{\infty}\frac{1}{(k x+x-y+1)^3}f\left(\frac{(-1)^k \left(3 x+(-1)^k ((2 k+5) x-2 y+2)-2 y-2\right)}{4 (k x+x-y+1)},\frac{x}{k x+x-y+1}\right) \\\hline
 (12,23,12) & \sum_{k=0}^{\infty}\frac{1}{(k x+x-y+1)^3}f\left(\frac{(k+2) x-y}{k x+x-y+1},\frac{x}{k x+x-y+1}\right) \\\hline
 (12,23,123) & \sum_{k=0}^{\infty}\frac{64}{\left| 2 x+(-1)^k (2 k+2 x-4 y+5)-1\right| ^3}f\left(\frac{6 x+(-1)^k (2 k+2 x-4 y+3)-3}{2 x+(-1)^k (2 k+2 x-4 y+5)-1},\frac{2 \left(2 x+(-1)^k-1\right)}{2 x+(-1)^k (2 k+2 x-4
   y+5)-1}\right) \\\hline
 (12,123,e) & \sum_{k=0}^{\infty}\frac{1}{\left| x+k (x-y-1)-2\right| ^3}f\left(\frac{(-1)^k \left(3 x+(-1)^k (5 x+2 k (x-y-1)-3 y-7)-y-1\right)}{4 (x+k (x-y-1)-2)},\frac{-x+y+1}{-x k+y k+k-x+2}\right) \\\hline
 (12,123,12) & \sum_{k=0}^{\infty}\frac{1}{(-x+k (-x+y+1)+2)^3}f\left(\frac{(k+2) (x-1)-(k+1) y}{x+k (x-y-1)-2},\frac{-x+y+1}{-x k+y k+k-x+2}\right) \\\hline
 (12,123,123) & \sum_{k=0}^{\infty}\frac{64}{\left| -2 x+(-1)^k (2 k-2 x-2 y+7)+2 y+1\right| ^3}f\left(\frac{-6 x+(-1)^k (2 k-2 x-2 y+5)+6 y+3}{-2 x+(-1)^k (2 k-2 x-2 y+7)+2 y+1},-\frac{2 \left(-2 x+(-1)^k+2 y+1\right)}{2 x-2
   y+(-1)^k (-2 k+2 x+2 y-7)-1}\right) \\\hline
 (12,132,e) & \sum_{k=0}^{\infty}\frac{1}{(k+x-(k+1) y+1)^3}f\left(\frac{(-1)^k \left(-2 x+3 y+(-1)^k (-2 x+2 k (y-1)+5 y-5)+1\right)}{4 (k+1) y-4 (k+x+1)},\frac{y-1}{-x+k (y-1)+y-1}\right) \\\hline
 (12,132,12) & \sum_{k=0}^{\infty}\frac{1}{(k+x-(k+1) y+1)^3}f\left(\frac{k+x-(k+2) y+1}{k+x-(k+1) y+1},\frac{y-1}{-x+k (y-1)+y-1}\right) \\\hline
 (12,132,123) & \sum_{k=0}^{\infty}\frac{64}{\left| (-1)^k (2 k+4 x-2 y+3)-2 y+1\right| ^3}f\left(\frac{(-1)^k (2 k+4 x-2 y+1)-6 y+3}{(-1)^k (2 k+4 x-2 y+3)-2 y+1},-\frac{2 \left(-2 y+(-1)^k+1\right)}{(-1)^{k+1} (2 k+4 x-2
   y+3)+2 y-1}\right) \\
\end{array}$}$$

$$\scalebox{0.9}{$\begin{array}{c|c}
 (13,e,13) & \sum_{k=0}^{\infty}\frac{1}{(k x+y+1)^3}f\left(1-\frac{x}{k x+y+1},1-\frac{1}{k x+y+1}\right) \\\hline
 (13,e,123) & \sum_{k=0}^{\infty}\frac{1}{(k x+y+1)^3}f\left(1-\frac{x}{k x+y+1},\frac{(-1)^k \left(x+2 y+(-1)^k (2 k x-x+2 y+2)-2\right)}{4 (k x+y+1)}\right) \\\hline
 (13,e,132) & \sum_{k=0}^{\infty}\frac{64}{\left| 2 x+(-1)^k (2 k-2 x+4 y+5)-1\right| ^3}f\left(\frac{-2 x+(-1)^k (2 k-2 x+4 y+3)+1}{2 x+(-1)^k (2 k-2 x+4 y+5)-1},\frac{2 x+(-1)^k (2 k-2 x+4 y+1)-1}{2 x+(-1)^k (2 k-2 x+4
   y+5)-1}\right) \\\hline
 (13,12,13) & \sum_{k=0}^{\infty}\frac{1}{(-x k+y k+k+y+1)^3}f\left(\frac{-x k+y k+k+x}{-x k+y k+k+y+1},1-\frac{1}{-x k+y k+k+y+1}\right) \\\hline
 (13,12,123) & \sum_{k=0}^{\infty}\frac{1}{(-x k+y k+k+y+1)^3}f\left(\frac{-x k+y k+k+x}{-x k+y k+k+y+1},\frac{x+(-1)^{k+1} (x-3 y+1)+y+2 k (-x+y+1)+1}{4 (y+k (-x+y+1)+1)}\right) \\\hline
 (13,12,132) & \sum_{k=0}^{\infty}\frac{64}{\left| -2 x+2 y+(-1)^k (2 k+2 x+2 y+3)+1\right| ^3}f\left(\frac{2 x-2 y+(-1)^k (2 k+2 x+2 y+1)-1}{-2 x+2 y+(-1)^k (2 k+2 x+2 y+3)+1},\frac{-2 x+2 y+(-1)^k (2 k+2 x+2 y-1)+1}{-2 x+2
   y+(-1)^k (2 k+2 x+2 y+3)+1}\right) \\\hline
 (13,13,13) & \sum_{k=0}^{\infty}-\frac{1}{(x+k (y-1)-2)^3}f\left(\frac{x+k (y-1)-y-1}{x+k (y-1)-2},1+\frac{1}{x+k (y-1)-2}\right) \\\hline
 (13,13,123) & \sum_{k=0}^{\infty}\frac{1}{\left| x+k (y-1)-2\right| ^3}f\left(\frac{x+k (y-1)-y-1}{x+k (y-1)-2},\frac{(-1)^k \left(2 x+(-1)^k (2 x+2 k (y-1)-y-3)+y-1\right)}{4 (x+k (y-1)-2)}\right) \\\hline
 (13,13,132) & \sum_{k=0}^{\infty}\frac{64}{\left| -2 y+(-1)^k (2 k-4 x+2 y+7)+1\right| ^3}f\left(\frac{2 y+(-1)^k (2 k-4 x+2 y+5)-1}{-2 y+(-1)^k (2 k-4 x+2 y+7)+1},\frac{-2 y+(-1)^k (2 k-4 x+2 y+3)+1}{-2 y+(-1)^k (2 k-4 x+2
   y+7)+1}\right) \\\hline
 (13,23,13) & \sum_{k=0}^{\infty}\frac{1}{(k x+x-y+1)^3}f\left(1-\frac{x}{k x+x-y+1},1+\frac{1}{-(k+1) x+y-1}\right) \\\hline
 (13,23,123) & \sum_{k=0}^{\infty}\frac{1}{(k x+x-y+1)^3}f\left(1-\frac{x}{k x+x-y+1},\frac{(-1)^k \left(3 x+(-1)^k (2 k x+x-2 y+2)-2 y-2\right)}{4 (k x+x-y+1)}\right) \\\hline
 (13,23,132) & \sum_{k=0}^{\infty}\frac{64}{\left| 2 x+(-1)^k (2 k+2 x-4 y+5)-1\right| ^3}f\left(\frac{-2 x+(-1)^k (2 k+2 x-4 y+3)+1}{2 x+(-1)^k (2 k+2 x-4 y+5)-1},\frac{2 x+(-1)^k (2 k+2 x-4 y+1)-1}{2 x+(-1)^k (2 k+2 x-4
   y+5)-1}\right) \\\hline
 (13,123,13) & \sum_{k=0}^{\infty}\frac{1}{(-x+k (-x+y+1)+2)^3}f\left(\frac{k (x-y-1)+y-1}{x+k (x-y-1)-2},1+\frac{1}{x+k (x-y-1)-2}\right) \\\hline
 (13,123,123) & \sum_{k=0}^{\infty}\frac{1}{\left| x+k (x-y-1)-2\right| ^3}f\left(\frac{k (x-y-1)+y-1}{x+k (x-y-1)-2},\frac{(-1)^k \left(3 x-y+(-1)^k (x+2 k (x-y-1)+y-3)-1\right)}{4 (x+k (x-y-1)-2)}\right) \\\hline
 (13,123,132) & \sum_{k=0}^{\infty}\frac{64}{\left| -2 x+(-1)^k (2 k-2 x-2 y+7)+2 y+1\right| ^3}f\left(\frac{2 x+(-1)^k (2 k-2 x-2 y+5)-2 y-1}{-2 x+(-1)^k (2 k-2 x-2 y+7)+2 y+1},\frac{-2 x+(-1)^k (2 k-2 x-2 y+3)+2 y+1}{-2
   x+(-1)^k (2 k-2 x-2 y+7)+2 y+1}\right) \\\hline
 (13,132,13) & \sum_{k=0}^{\infty}\frac{1}{(k+x-(k+1) y+1)^3}f\left(\frac{-y k+k+x}{k+x-(k+1) y+1},1+\frac{1}{-x+k (y-1)+y-1}\right) \\\hline
 (13,132,123) & \sum_{k=0}^{\infty}\frac{1}{(k+x-(k+1) y+1)^3}f\left(\frac{-y k+k+x}{k+x-(k+1) y+1},\frac{(-1)^k \left(-2 x+3 y+(-1)^k (-2 x+2 k (y-1)+y-1)+1\right)}{4 (k+1) y-4 (k+x+1)}\right) \\\hline
 (13,132,132) & \sum_{k=0}^{\infty}\frac{64}{\left| (-1)^k (2 k+4 x-2 y+3)-2 y+1\right| ^3}f\left(\frac{(-1)^k (2 k+4 x-2 y+1)+2 y-1}{(-1)^k (2 k+4 x-2 y+3)-2 y+1},\frac{(-1)^k (2 k+4 x-2 y-1)-2 y+1}{(-1)^k (2 k+4 x-2 y+3)-2
   y+1}\right) \\
\end{array}$}$$

$$\scalebox{0.85}{$\begin{array}{c|c}
 (23,e,e) & \sum_{k=0}^{\infty}\frac{64}{\left| 2 x+(-1)^k (2 k-2 x+4 y+5)-1\right| ^3}f\left(\frac{4 (-1)^k}{2 x+(-1)^k (2 k-2 x+4 y+5)-1},-\frac{2 \left(-2 x+(-1)^k+1\right)}{-2 x+(-1)^{k+1} (2 k-2 x+4
   y+5)+1}\right) \\\hline
 (23,e,23) & \sum_{k=0}^{\infty}\frac{1}{(k x+y+1)^3}f\left(\frac{1}{k x+y+1},\frac{1-x}{k x+y+1}\right) \\\hline
 (23,e,132) & \sum_{k=0}^{\infty}\frac{1}{(k x+y+1)^3}f\left(\frac{(-1)^k \left(-x-2 y+(-1)^k (2 k x+x+2 y+2)+2\right)}{4 (k x+y+1)},\frac{(-1)^k \left(-x-2 y+(-1)^k (2 k x-3 x+2 y+2)+2\right)}{4 (k x+y+1)}\right) \\\hline
 (23,12,e) & \sum_{k=0}^{\infty}\frac{64}{\left| -2 x+2 y+(-1)^k (2 k+2 x+2 y+3)+1\right| ^3}f\left(\frac{4 (-1)^k}{-2 x+2 y+(-1)^k (2 k+2 x+2 y+3)+1},\frac{2 \left(2 x+(-1)^k-2 y-1\right)}{-2 x+2 y+(-1)^k (2 k+2 x+2
   y+3)+1}\right) \\\hline
 (23,12,23) & \sum_{k=0}^{\infty}\frac{1}{(-x k+y k+k+y+1)^3}f\left(\frac{1}{-x k+y k+k+y+1},\frac{x-y}{-x k+y k+k+y+1}\right) \\\hline
 (23,12,132) & \sum_{k=0}^{\infty}\frac{1}{(-x k+y k+k+y+1)^3}f\left(\frac{-x+(-1)^k (x-3 y+1)+3 y+2 k (-x+y+1)+3}{4 (y+k (-x+y+1)+1)},\frac{(-1)^k \left(x-3 y+(-1)^{k+1} (-3 x+2 k (x-y-1)+y+1)+1\right)}{4 (y+k (-x+y+1)+1)}\right) \\\hline
 (23,13,e) & \sum_{k=0}^{\infty}\frac{64}{\left| -2 y+(-1)^k (2 k-4 x+2 y+7)+1\right| ^3}f\left(\frac{4 (-1)^k}{-2 y+(-1)^k (2 k-4 x+2 y+7)+1},\frac{2 \left(2 y+(-1)^k-1\right)}{-2 y+(-1)^k (2 k-4 x+2 y+7)+1}\right) \\\hline
 (23,13,23) & \sum_{k=0}^{\infty}\frac{1}{(-y k+k-x+2)^3}f\left(\frac{1}{-y k+k-x+2},\frac{y}{-y k+k-x+2}\right) \\\hline
 (23,13,132) & \sum_{k=0}^{\infty}\frac{1}{\left| x+k (y-1)-2\right| ^3}f\left(\frac{(-1)^k \left(-2 x-y+(-1)^k (2 x+2 k (y-1)+y-5)+1\right)}{4 (x+k (y-1)-2)},\frac{(-1)^k \left(-2 x+(-1)^k (2 x+2 k (y-1)-3 y-1)-y+1\right)}{4 (x+k
   (y-1)-2)}\right) \\\hline
 (23,23,e) & \sum_{k=0}^{\infty}\frac{64}{\left| 2 x+(-1)^k (2 k+2 x-4 y+5)-1\right| ^3}f\left(\frac{4 (-1)^k}{2 x+(-1)^k (2 k+2 x-4 y+5)-1},\frac{2 \left(-2 x+(-1)^k+1\right)}{2 x+(-1)^k (2 k+2 x-4 y+5)-1}\right) \\\hline
 (23,23,23) & \sum_{k=0}^{\infty}\frac{1}{(k x+x-y+1)^3}f\left(\frac{1}{k x+x-y+1},\frac{1-x}{k x+x-y+1}\right) \\\hline
 (23,23,132) & \sum_{k=0}^{\infty}\frac{1}{(k x+x-y+1)^3}f\left(\frac{(-1)^k \left(-3 x+(-1)^k ((2 k+3) x-2 y+2)+2 y+2\right)}{4 (k x+x-y+1)},\frac{(-1)^k \left(-3 x+(-1)^k ((2 k-1) x-2 y+2)+2 y+2\right)}{4 (k x+x-y+1)}\right) \\\hline
 (23,123,e) & \sum_{k=0}^{\infty}\frac{64}{\left| -2 x+(-1)^k (2 k-2 x-2 y+7)+2 y+1\right| ^3}f\left(\frac{4 (-1)^k}{-2 x+(-1)^k (2 k-2 x-2 y+7)+2 y+1},\frac{-4 x-2 (-1)^k+4 y+2}{2 x-2 y+(-1)^k (-2 k+2 x+2 y-7)-1}\right)
   \\\hline
 (23,123,23) & \sum_{k=0}^{\infty}\frac{1}{(-x+k (-x+y+1)+2)^3}f\left(\frac{1}{-x k+y k+k-x+2},\frac{y-x}{x+k (x-y-1)-2}\right) \\\hline
 (23,123,132) & \sum_{k=0}^{\infty}\frac{1}{\left| x+k (x-y-1)-2\right| ^3}f\left(\frac{(-1)^k \left(-3 x+(-1)^k (3 x+2 k (x-y-1)-y-5)+y+1\right)}{4 (x+k (x-y-1)-2)},\frac{(-1)^k \left(-3 x+y+(-1)^k (-x+2 k (x-y-1)+3 y-1)+1\right)}{4 (x+k
   (x-y-1)-2)}\right) \\\hline
 (23,132,e) & \sum_{k=0}^{\infty}\frac{64}{\left| (-1)^k (2 k+4 x-2 y+3)-2 y+1\right| ^3}f\left(\frac{4 (-1)^k}{(-1)^k (2 k+4 x-2 y+3)-2 y+1},-\frac{2 \left(2 y+(-1)^k-1\right)}{(-1)^{k+1} (2 k+4 x-2 y+3)+2 y-1}\right) \\\hline
 (23,132,23) & \sum_{k=0}^{\infty}\frac{1}{(k+x-(k+1) y+1)^3}f\left(\frac{1}{k+x-(k+1) y+1},\frac{y}{k+x-(k+1) y+1}\right) \\\hline
 (23,132,132) & \sum_{k=0}^{\infty}\frac{1}{(k+x-(k+1) y+1)^3}f\left(\frac{(-1)^k \left(2 x-3 y+(-1)^k (-2 x+2 k (y-1)+3 y-3)-1\right)}{4 (k+1) y-4 (k+x+1)},\frac{(-1)^k \left(2 x-3 y+(-1)^{k+1} (2 x-2 k (y-1)+y-1)-1\right)}{4 (k+1) y-4
   (k+x+1)}\right) \\
\end{array}$}$$

$$\scalebox{0.85}{$\begin{array}{c|c}
 (123,e,13) & \sum_{k=0}^{\infty}\frac{64}{\left| 2 x+(-1)^k (2 k-2 x+4 y+5)-1\right| ^3}f\left(\frac{6 x+(-1)^k (2 k-2 x+4 y+3)-3}{2 x+(-1)^k (2 k-2 x+4 y+5)-1},\frac{2 x+(-1)^k (2 k-2 x+4 y+1)-1}{2 x+(-1)^k (2 k-2 x+4
   y+5)-1}\right) \\\hline
 (123,e,23) & \sum_{k=0}^{\infty}\frac{1}{(k x+y+1)^3}f\left(\frac{(-1)^k \left(x+2 y+(-1)^k (2 k x+3 x+2 y+2)-2\right)}{4 (k x+y+1)},\frac{(-1)^k \left(x+2 y+(-1)^k (2 k x-x+2 y+2)-2\right)}{4 (k x+y+1)}\right) \\\hline
 (123,e,132) & \sum_{k=0}^{\infty}\frac{1}{(k x+y+1)^3}f\left(\frac{k x+x+y}{k x+y+1},1-\frac{1}{k x+y+1}\right) \\\hline
 (123,12,13) & \sum_{k=0}^{\infty}\frac{64}{\left| -2 x+2 y+(-1)^k (2 k+2 x+2 y+3)+1\right| ^3}f\left(\frac{-6 x+6 y+(-1)^k (2 k+2 x+2 y+1)+3}{-2 x+2 y+(-1)^k (2 k+2 x+2 y+3)+1},\frac{-2 x+2 y+(-1)^k (2 k+2 x+2 y-1)+1}{-2 x+2
   y+(-1)^k (2 k+2 x+2 y+3)+1}\right) \\\hline
 (123,12,23) & \sum_{k=0}^{\infty}\frac{1}{(-x k+y k+k+y+1)^3}f\left(-\frac{(-1)^k \left(x-3 y+(-1)^k (3 x+2 k (x-y-1)-5 (y+1))+1\right)}{4 (y+k (-x+y+1)+1)},\frac{x+(-1)^{k+1} (x-3 y+1)+y+2 k (-x+y+1)+1}{4 (y+k (-x+y+1)+1)}\right) \\\hline
 (123,12,132) & \sum_{k=0}^{\infty}\frac{1}{(-x k+y k+k+y+1)^3}f\left(\frac{k-(k+1) x+(k+2) y+1}{y+k (-x+y+1)+1},1-\frac{1}{-x k+y k+k+y+1}\right) \\\hline
 (123,13,13) & \sum_{k=0}^{\infty}\frac{64}{\left| -2 y+(-1)^k (2 k-4 x+2 y+7)+1\right| ^3}f\left(\frac{-6 y+(-1)^k (2 k-4 x+2 y+5)+3}{-2 y+(-1)^k (2 k-4 x+2 y+7)+1},\frac{-2 y+(-1)^k (2 k-4 x+2 y+3)+1}{-2 y+(-1)^k (2 k-4 x+2
   y+7)+1}\right) \\\hline
 (123,13,23) & \sum_{k=0}^{\infty}\frac{1}{\left| x+k (y-1)-2\right| ^3}f\left(\frac{(-1)^k \left(2 x+y+(-1)^k (2 x+2 k (y-1)+3 y-7)-1\right)}{4 (x+k (y-1)-2)},\frac{(-1)^k \left(2 x+(-1)^k (2 x+2 k (y-1)-y-3)+y-1\right)}{4 (x+k
   (y-1)-2)}\right) \\\hline
 (123,13,132) & \sum_{k=0}^{\infty}-\frac{1}{(x+k (y-1)-2)^3}f\left(\frac{x+k (y-1)+y-2}{x+k (y-1)-2},1+\frac{1}{x+k (y-1)-2}\right) \\\hline
 (123,23,13) & \sum_{k=0}^{\infty}\frac{64}{\left| 2 x+(-1)^k (2 k+2 x-4 y+5)-1\right| ^3}f\left(\frac{6 x+(-1)^k (2 k+2 x-4 y+3)-3}{2 x+(-1)^k (2 k+2 x-4 y+5)-1},\frac{2 x+(-1)^k (2 k+2 x-4 y+1)-1}{2 x+(-1)^k (2 k+2 x-4
   y+5)-1}\right) \\\hline
 (123,23,23) & \sum_{k=0}^{\infty}\frac{1}{(k x+x-y+1)^3}f\left(\frac{(-1)^k \left(3 x+(-1)^k ((2 k+5) x-2 y+2)-2 y-2\right)}{4 (k x+x-y+1)},\frac{(-1)^k \left(3 x+(-1)^k (2 k x+x-2 y+2)-2 y-2\right)}{4 (k x+x-y+1)}\right) \\\hline
 (123,23,132) & \sum_{k=0}^{\infty}\frac{1}{(k x+x-y+1)^3}f\left(\frac{(k+2) x-y}{k x+x-y+1},1+\frac{1}{-(k+1) x+y-1}\right) \\\hline
 (123,123,13) & \sum_{k=0}^{\infty}\frac{64}{\left| -2 x+(-1)^k (2 k-2 x-2 y+7)+2 y+1\right| ^3}f\left(\frac{-6 x+(-1)^k (2 k-2 x-2 y+5)+6 y+3}{-2 x+(-1)^k (2 k-2 x-2 y+7)+2 y+1},\frac{-2 x+(-1)^k (2 k-2 x-2 y+3)+2 y+1}{-2
   x+(-1)^k (2 k-2 x-2 y+7)+2 y+1}\right) \\\hline
 (123,123,23) & \sum_{k=0}^{\infty}\frac{1}{\left| x+k (x-y-1)-2\right| ^3}f\left(\frac{(-1)^k \left(3 x+(-1)^k (5 x+2 k (x-y-1)-3 y-7)-y-1\right)}{4 (x+k (x-y-1)-2)},\frac{(-1)^k \left(3 x-y+(-1)^k (x+2 k (x-y-1)+y-3)-1\right)}{4 (x+k
   (x-y-1)-2)}\right) \\\hline
 (123,123,132) & \sum_{k=0}^{\infty}\frac{1}{(-x+k (-x+y+1)+2)^3}f\left(\frac{(k+2) (x-1)-(k+1) y}{x+k (x-y-1)-2},1+\frac{1}{x+k (x-y-1)-2}\right) \\\hline
 (123,132,13) & \sum_{k=0}^{\infty}\frac{64}{\left| (-1)^k (2 k+4 x-2 y+3)-2 y+1\right| ^3}f\left(\frac{(-1)^k (2 k+4 x-2 y+1)-6 y+3}{(-1)^k (2 k+4 x-2 y+3)-2 y+1},\frac{(-1)^k (2 k+4 x-2 y-1)-2 y+1}{(-1)^k (2 k+4 x-2 y+3)-2
   y+1}\right) \\\hline
 (123,132,23) & \sum_{k=0}^{\infty}\frac{1}{(k+x-(k+1) y+1)^3}f\left(\frac{(-1)^k \left(-2 x+3 y+(-1)^k (-2 x+2 k (y-1)+5 y-5)+1\right)}{4 (k+1) y-4 (k+x+1)},\frac{(-1)^k \left(-2 x+3 y+(-1)^k (-2 x+2 k (y-1)+y-1)+1\right)}{4 (k+1) y-4
   (k+x+1)}\right) \\\hline
 (123,132,132) & \sum_{k=0}^{\infty}\frac{1}{(k+x-(k+1) y+1)^3}f\left(\frac{k+x-(k+2) y+1}{k+x-(k+1) y+1},1+\frac{1}{-x+k (y-1)+y-1}\right) \\
\end{array}$}$$

$$\scalebox{0.9}{$\begin{array}{c|c}
 (132,e,12) & \sum_{k=0}^{\infty}\frac{64}{\left| 2 x+(-1)^k (2 k-2 x+4 y+5)-1\right| ^3}f\left(\frac{-2 x+(-1)^k (2 k-2 x+4 y+3)+1}{2 x+(-1)^k (2 k-2 x+4 y+5)-1},-\frac{2 \left(-2 x+(-1)^k+1\right)}{-2 x+(-1)^{k+1} (2
   k-2 x+4 y+5)+1}\right) \\\hline
 (132,e,13) & \sum_{k=0}^{\infty}\frac{1}{(k x+y+1)^3}f\left(1-\frac{x}{k x+y+1},\frac{(-1)^k \left(-x-2 y+(-1)^k (2 k x-3 x+2 y+2)+2\right)}{4 (k x+y+1)}\right) \\\hline
 (132,e,123) & \sum_{k=0}^{\infty}\frac{1}{(k x+y+1)^3}f\left(1-\frac{x}{k x+y+1},\frac{1-x}{k x+y+1}\right) \\\hline
 (132,12,12) & \sum_{k=0}^{\infty}\frac{64}{\left| -2 x+2 y+(-1)^k (2 k+2 x+2 y+3)+1\right| ^3}f\left(\frac{2 x-2 y+(-1)^k (2 k+2 x+2 y+1)-1}{-2 x+2 y+(-1)^k (2 k+2 x+2 y+3)+1},\frac{2 \left(2 x+(-1)^k-2 y-1\right)}{-2 x+2
   y+(-1)^k (2 k+2 x+2 y+3)+1}\right) \\\hline
 (132,12,13) & \sum_{k=0}^{\infty}\frac{1}{(-x k+y k+k+y+1)^3}f\left(\frac{-x k+y k+k+x}{-x k+y k+k+y+1},\frac{(-1)^k \left(x-3 y+(-1)^{k+1} (-3 x+2 k (x-y-1)+y+1)+1\right)}{4 (y+k (-x+y+1)+1)}\right) \\\hline
 (132,12,123) & \sum_{k=0}^{\infty}\frac{1}{(-x k+y k+k+y+1)^3}f\left(\frac{-x k+y k+k+x}{-x k+y k+k+y+1},\frac{x-y}{-x k+y k+k+y+1}\right) \\\hline
 (132,13,12) & \sum_{k=0}^{\infty}\frac{64}{\left| -2 y+(-1)^k (2 k-4 x+2 y+7)+1\right| ^3}f\left(\frac{2 y+(-1)^k (2 k-4 x+2 y+5)-1}{-2 y+(-1)^k (2 k-4 x+2 y+7)+1},\frac{2 \left(2 y+(-1)^k-1\right)}{-2 y+(-1)^k (2 k-4 x+2
   y+7)+1}\right) \\\hline
 (132,13,13) & \sum_{k=0}^{\infty}\frac{1}{\left| x+k (y-1)-2\right| ^3}f\left(\frac{x+k (y-1)-y-1}{x+k (y-1)-2},\frac{(-1)^k \left(-2 x+(-1)^k (2 x+2 k (y-1)-3 y-1)-y+1\right)}{4 (x+k (y-1)-2)}\right) \\\hline
 (132,13,123) & \sum_{k=0}^{\infty}-\frac{1}{(x+k (y-1)-2)^3}f\left(\frac{x+k (y-1)-y-1}{x+k (y-1)-2},\frac{y}{-y k+k-x+2}\right) \\\hline
 (132,23,12) & \sum_{k=0}^{\infty}\frac{64}{\left| 2 x+(-1)^k (2 k+2 x-4 y+5)-1\right| ^3}f\left(\frac{-2 x+(-1)^k (2 k+2 x-4 y+3)+1}{2 x+(-1)^k (2 k+2 x-4 y+5)-1},\frac{2 \left(-2 x+(-1)^k+1\right)}{2 x+(-1)^k (2 k+2 x-4
   y+5)-1}\right) \\\hline
 (132,23,13) & \sum_{k=0}^{\infty}\frac{1}{(k x+x-y+1)^3}f\left(1-\frac{x}{k x+x-y+1},\frac{(-1)^k \left(-3 x+(-1)^k ((2 k-1) x-2 y+2)+2 y+2\right)}{4 (k x+x-y+1)}\right) \\\hline
 (132,23,123) & \sum_{k=0}^{\infty}\frac{1}{(k x+x-y+1)^3}f\left(1-\frac{x}{k x+x-y+1},\frac{1-x}{k x+x-y+1}\right) \\\hline
 (132,123,12) & \sum_{k=0}^{\infty}\frac{64}{\left| -2 x+(-1)^k (2 k-2 x-2 y+7)+2 y+1\right| ^3}f\left(\frac{2 x+(-1)^k (2 k-2 x-2 y+5)-2 y-1}{-2 x+(-1)^k (2 k-2 x-2 y+7)+2 y+1},\frac{-4 x-2 (-1)^k+4 y+2}{2 x-2 y+(-1)^k (-2 k+2
   x+2 y-7)-1}\right) \\\hline
 (132,123,13) & \sum_{k=0}^{\infty}\frac{1}{\left| x+k (x-y-1)-2\right| ^3}f\left(\frac{k (x-y-1)+y-1}{x+k (x-y-1)-2},\frac{(-1)^k \left(-3 x+y+(-1)^k (-x+2 k (x-y-1)+3 y-1)+1\right)}{4 (x+k (x-y-1)-2)}\right) \\\hline
 (132,123,123) & \sum_{k=0}^{\infty}\frac{1}{(-x+k (-x+y+1)+2)^3}f\left(\frac{k (x-y-1)+y-1}{x+k (x-y-1)-2},\frac{y-x}{x+k (x-y-1)-2}\right) \\\hline
 (132,132,12) & \sum_{k=0}^{\infty}\frac{64}{\left| (-1)^k (2 k+4 x-2 y+3)-2 y+1\right| ^3}f\left(\frac{(-1)^k (2 k+4 x-2 y+1)+2 y-1}{(-1)^k (2 k+4 x-2 y+3)-2 y+1},-\frac{2 \left(2 y+(-1)^k-1\right)}{(-1)^{k+1} (2 k+4 x-2 y+3)+2
   y-1}\right) \\\hline
 (132,132,13) & \sum_{k=0}^{\infty}\frac{1}{(k+x-(k+1) y+1)^3}f\left(\frac{-y k+k+x}{k+x-(k+1) y+1},\frac{(-1)^k \left(2 x-3 y+(-1)^{k+1} (2 x-2 k (y-1)+y-1)-1\right)}{4 (k+1) y-4 (k+x+1)}\right) \\\hline
 (132,132,123) & \sum_{k=0}^{\infty}\frac{1}{(k+x-(k+1) y+1)^3}f\left(\frac{-y k+k+x}{k+x-(k+1) y+1},\frac{y}{k+x-(k+1) y+1}\right) \\
\end{array}$}$$

\end{document}